\documentclass[11pt]{amsart}
\usepackage[active]{srcltx}
\usepackage{mathrsfs}
\usepackage[T1]{fontenc}

\usepackage{latexsym,enumerate}
\usepackage{amsmath,amssymb,amsthm,amsfonts,latexsym}
\usepackage[bookmarks]{hyperref}

\hypersetup{backref, colorlinks=true}
\hypersetup{backref, colorlinks=true, colorlinks   = true,
	urlcolor     = blue, linkcolor = blue, citecolor   = magenta
}
\def\adh#1{\overline{#1}}

\let\=\partial

\newtheorem {pro}{Proposition}[section]
\newtheorem {thm}[pro]{Theorem}%[section]
\newtheorem{lem}[pro]{Lemma}

\theoremstyle{definition}
 \newtheorem {rem}[pro]{Remark}%[section]
\newtheorem {dfn}[pro]{Definition}%[section]

\newtheorem {exa}[pro]{Example}

\newtheorem {step}{Step}
\newtheorem {claim}{Claim}

\newcommand{\nen}{{\bar{n}}}
\newcommand{\tim}{{t\in \R^m}}
\newcommand{\nub}{\mathbf{\nu}}

\newcommand{\gab}{\mathbf{\gamma}}
\newcommand{\gub}{\gab_\nub}
\newcommand{\dvv}{{\mbox{div}}}

\newcommand{\s}{\mathcal{S}}

\newcommand{\cbt}{\tilde{\bf C}}
\newcommand{\cbf}{{\bf C}}

\newcommand{\jac}{\mbox{jac}\,}
\newcommand{\tra}{\mathbf{tr}}

\newcommand{\R}{\mathbb{R}}

\newcommand{\N}{\mathbb{N}}
\newcommand{\cc}{\mathscr{C}}\newcommand{\C}{\mathcal{C}}

\newcommand{\et}{\quad \mbox{and} \quad }

\newcommand{\hn}{\mathcal{H}}

\newcommand{\omd}{{\Omega}}

\newcommand{\wbf}{\mathbf{W}}

\newcommand{\lbf}{\mathbf{L}}

\newcommand{\D}{\mathcal{D}}

\newcommand{\ep}{\varepsilon}

\newcommand{\pa}{\partial}

\newcommand{\hh}{\mathcal{V}}

\newcommand{\bou}{\mathbf{B}}
\newcommand{\sph}{\mathbf{S}}
\newcommand{\orn}{{0_{\R^n}}}
\newcommand{\supp}{\mbox{\rm supp}}
\newcommand{\xo}{{x_0}}
\newcommand{\stt}{\,\tilde{*}\,}

\title[]{$W^{1,p}$ priori estimates for solutions of linear elliptic PDEs on subanalytic domains}

\makeatletter
 \thanks{Research partially supported
by the NCN grant  2021/43/B/ST1/02359.}
\@namedef{subjclassname@2020}{%
\textup{2020} Mathematics Subject Classification}
\makeatletter

\@addtoreset{equation}{section}

\makeatother

\author[ G. Valette]{ Guillaume Valette}

\address[G. Valette]{Instytut Matematyki Uniwersytetu
Jagiello\'nskiego, ul. S. \L ojasiewicza 6, Krak\'ow, Poland}\email{guillaume.valette@im.uj.edu.pl}

\keywords{Elliptic PDE, subanalytic domain, singular domain,  a priori estimate, tangent cone, stratification, \L ojasiewicz's inequality}

\subjclass[2020]{35B45, 35A21, 32B20, 14P10}

\begin{document} \maketitle
\begin{abstract}
	We prove a priori estimates for solutions of order $2$ linear elliptic PDEs in divergence form on subanalytic domains. More precisely, we study the solutions of a strongly elliptic equation $Lu=f$, with $f\in L^2(\omd)$ and $Lu=\mbox{div} (A(x) \nabla u)$, and, 	given a bounded subanalytic domain $\omd$, possibly admitting non metrically conical singularities within its boundary, we provide explicit conditions on the tangent cone of the singularities of the boundary which ensure that  $||u||_{ W^{1,p}(\omd)}\le C||f||_{L^2(\omd)}$,  for some $p>2$. 		 The number $p$ depends on the geometry of the singularities of $\delta \omd$, but not on $u$.

%	Given a bounded subanalytic domain $\omd$, possibly admitting non metrically conical singularities within its boundary, we give explicit conditions on the tangent cone of the singularities of the boundary which ensure that if $u\in W^{1,2}_0(\omd)$ satisfies $Lu=f$, with $f\in L^2(\omd)$ and 	$L	u=\mbox{div} (A(x) \nabla u)$ then $u\in W^{1,p}(\omd)$ for some $p>2$. Here $\hn^k$ denotes the Hausdorff measure  $\delta \omd$ the boundary of $\omd$.
%		 The number $p$ depends on the geometry of the singularities of $\delta \omd$, but not on $u$. 
	\end{abstract}
	\section{Introduction}
	As well-known, the singularities of the boundary of a domain significantly affect the level of integrability of the solution of an elliptic partial differential equation. In this article, given a bounded subanalytic domain $
	\omd$ of $\R^n$, we prove a priori estimates for solutions of elliptic equations in divergence form
	\begin{equation}\label{eq_L}
		L	u:=\mbox{div} (A(x) \nabla u) 	=f,
	\end{equation}
	where $A\in W^{1,\infty}(\omd)^{n\times n}$ satisfies   for some $\lambda_0>0$,
	\begin{equation}\label{eq_ellipticity}
		\lambda_0|\xi|^2\le \sum_{i,j=1} ^n A_{ij}(x)\xi_i\xi_j,\qquad \forall \xi \in \R^n, \;\forall x\in  \omd .
	\end{equation}
 
 Subanalytic sets  are the  images of semi-analytic sets under proper analytic morphisms  \cite{bm,livre},  semi-analytic sets being the subsets of $\R^n$  locally defined by finitely many equalities and inequalities on  analytic functions.
 % The subanalytic category thus comprises all the subsets of $\R^n$ defined by finitely many inequalities on  polynomials. 
  Subanalytic domains may admit very singular boundary and it is easy to produce examples where the cone property fails. Recent works on the theory of Sobolev spaces of subanalytic sets \cite{poincfried, poincwirt, trace, lprime, gupel} however underline that they constitute a well-behaved category to investigate PDEs, and the level of generality that they offer makes them very attractive for this study.

 We investigate in this article the $W^{1,p}$ regularity of a solution $u\in W^{1,2}(\omd)$ of (\ref{eq_L}) under Dirichlet boundary condition. We show that $u\in W^{1,p}(\omd)$ for some $p>2$ under an assumption on the tangent cone of the boundary or the complement of $\omd$ (Theorem \ref{main}). The problem of the continuity of the solution of Poisson equation was studied in \cite{kai} using the Perron method, who showed that the solution is continuous under some slightly weaker condition on the tangent cone of the complement, that was also proved  to be necessary.

 Our approach rather focuses on the theory of Sobolev spaces.
 We generalize well-known facts on Lipschitz domains \cite{jerker}, on which the theory is already intricate. If it is true that singularities have a significant effect on the regularity of solutions, it is on the other hand an interesting problem to study the extent to which the geometry of the singularities is related to the singularity of the solution, and the present article emphasizes that  results achieved by singularists during the seven last decades can provide information on this issue.   For instance, the number $p>2$ provided by Theorem \ref{main} is closely related to the geometry of the domain and inward horns that may arise.
  The notion of tangent cone, introduced by singularists, has thus  much to do with this issue, and many results of singularity theory, like \L ojasiewicz's inequality \cite{loj}, existence of Verdier regular stratifications \cite{ver}, or normal pseudo-flatness of stratified spaces \cite{hir, orrtro} (more precisely, the closely related formulas (\ref{eq_lim_f_t}) and (\ref{eq_trivialite_cone_normal})), prove useful for our purpose. Uniform finiteness \cite[Corollary $1.8.12$]{livre} \cite[Theorem $2.9$]{cos}  provides bounds for $L^p$ norms (see section \ref{sect_pf_vshort}).
  
   The purpose of the present article is therefore also  to stress the interplay between singularity theory and the theory of PDEs. One advantage of working in the subanalytic category is that  it includes all the sets defined by finitely many inequalities on polynomials. That the condition provided by our theorem only bears on tangent cones makes it possible to give  effective local criteria to identify  points of the boundary at which the solution is regular, as the tangent cone often only depends on the initial form of the equation of $\delta \omd$.  We give an example of application to the  study of the function $u:\bou_{\R^n}(0,1)\to \R$, where $\bou_{\R^n}(0,1)$ is the ball of radius $1$ in $\R^n$, realizing the minimum
   \begin{equation}\label{eq_argmin}
   	 \min \; \{ ||\nabla u||_{L^2(\bou_{\R^n}(0,1))}: u(x)=g(x) \mbox{ at all $x\in \adh{\bou_{\R^n}(0,1)}$ satisfying } P(x)=0\},
   \end{equation}
   when $P$ is a polynomial and $g\in W^{2,2}(\bou_{\R^n}(0,1))$
    (see Examples \ref{exa1}, \ref{exa2}).

	\section{Some notations and conventions}
	Throughout this article, the letter $\omd$ stands for a bounded subanalytic open subset of $\R^n$. We will use the following notations.
	 
	\begin{itemize}
	 \item 	 $\bou_A(x,\ep)$, open  ball in $A\subset \R^n$ of radius $\ep$   centered  at $x$ (for the  euclidean distance).  To simplify notations, we will however
	omit the subscript $A$ when obvious from the context.

		\item $\sph(x,\ep)$, sphere    of radius  $\ep>0$ centered at   $x\in  \R^n$. As customary, $\sph(0_{\R^n},1)$ will however be denoted $\sph^{n-1}$, $0_{\R^n}$ being the origin of $\R^n$.
		 
		 	\item 	$x\cdot y$  and $|x|$,  euclidean inner product of $x$ and $y$ and euclidean norm of $x$.
		 
		 \item 	$\adh A$, closure of a set $A\subset \R^n$.
		 \item $\pa \omd$, regular locus of $\delta \omd:= \adh \omd \setminus \omd$ (section \ref{sect_sub}).

	\item $D_xh$   derivative of a mapping $h$ at a point $x$.  As customary, we however write $h'(x)$ whenever $h$ is a one-variable function.

	\item 	$r \cdot A$, image of $A\subset \R^n$ under the mapping $x\mapsto rx$, if $r\in \R$.

		\item $<u,v>$, $L^2$-inner product of $\omd$, i.e. $\int_\omd uv$. %We write $<u,v>_U$ for the $L^2$-inner product of $U$, if $U$ is a manifold
	
	\item 	$\nabla u$ and $\dvv\, v$, gradient of $u:\omd \to \R$ and divergence of $v:\omd\to \R^n$.

%	$\dvv \beta$ : divergence of a  vector field $\beta$ on $\omd$

	\item 	Given   $p\in [1,\infty)$, $W^{1,p}(\omd)$ stands for the Sobolev space of $\omd$, i.e. $$W^{1,p}(\omd):= \{u\in L^p(\omd),\; |\nabla u| \in L^p(\omd)\},$$ and, given an open subset $Z$ of $\pa \omd$,
	$$W^{1,p}(\omd,Z):= \{u\in W^{1,p}(\omd), \tra_{\pa \omd} u=0\; \mbox{on $Z$}\},$$
	where $\tra_{\pa \omd} $ is the trace operator defined in section \ref{sect_PDEs}. Since the elements of  $W^{1,p}(\omd,\pa \omd)$ can  be extended by $0$ to elements of $W^{1,p}(\R^n)$ \cite[Proposition $4.5$]{lprime}, we will sometimes regard them as defined on $\R^n$. %We will write $W_0^{1,p}(\omd)$ for the closure of $\cc_0^\infty(\omd)$ in $W^{1,p}(\omd)$.
% 	Given a measurable function $\rho$ on $\omd$, we  set
% 	$$||u||_{L^p_\rho(\omd)}:=\left(\int_\omd |u|^p\rho\right)^{1/p}.$$

	\item 	$\supp\, u$,  support of a distribution $u$ on $\omd$.

	\item 	$\s_n$, set of globally subanalytic subsets of $\R^n$ (section \ref{sect_sub}).

\item 	$\D^+(S)$, set of positive definable continuous functions on $S\subset \R^n$ (section \ref{sect_sub}).
	
	\item 	$A_t$, fiber of $A\in \s_{m+n}$ at $\tim$ (section \ref{sect_pf_vshort}).

	  	\item 	$\cbf_t(A)$ (resp. $\cbf^S_t(A)$), tangent cone (resp. normal cone) of $A\subset \R^n$ at $t\in A$ (section \ref{sect_a_priori}, resp. section \ref{sect_normal_cone}). 	$\cbt_t(A)$ (resp. $\cbt^S_t(A)$) will stand for the link of the tangent (resp. normal) cone (section \ref{sect_a_priori}, resp. section \ref{sect_normal_cone}).
	  
	\item 	$d(x,S)$,  euclidean distance from $x\in \R^n$ to  $S\subset \R^n$. The function $x\mapsto d(x,S)$ is  denoted $d(\cdot, S)$.

	 	\item   If $S$ is a definable $\cc^\infty$ submanifold of $\R^n$, then there is a definable  neighborhood $U_S$ of $S$ on which  $d(\cdot,S)^2$ is $\cc^\infty$ and there is a $\cc^\infty$ retraction $\pi_S:U_S\to S$ which assigns to every $x\in U_S$ the unique point that realizes $d(x,S)$, i.e. $|x-\pi_S(x)|=d(x,S)$ (see for instance \cite[Proposition $2.4.1$]{livre}).
	  We then set for $\delta\in \D^+(S)$
	  \begin{equation}\label{eq_tubular}
	  	U_S^\delta :=\{x\in \R^n :d(x,S)<\delta( \pi_S(x))\}.
	  \end{equation}
We will also write $U_S^\delta$ when $\delta$ is a positive number, identifying it with the constant function.
	
		\item $\hn^k$ and $d_\hn(\cdot, \cdot)$, $k$-dimensional Hausdorff measure and Hausdorff distance.
	
	 	\item  Given two nonnegative functions $\xi$ and $\zeta$ on a set $E$ as well as a subset $Z$ of $E$, we write ``$\xi\lesssim \zeta$ on $Z$'' or  ``$\xi(x)\lesssim \zeta(x)$ for $x\in Z$'' when there is a constant $C$ such that $\xi(x) \le C\zeta(x)$ for all $x\in Z$. %We will however sometimes specify  constants, when useful.

	\item Given a  mapping $h:M\to M'$ between differentiable submanifolds,  $M\subset\R^n$ and $M'\subset \R^k$, differentiable at $x\in M$, with $\dim M \ge \dim M'$, we write $\jac h(x)$ for the (generalized) jacobian  $\sqrt{\det D_xh D_xh^\mathbf{t}}$, where $D_x h^\mathbf{t}$ denotes the transposite of $D_xh$. We   recall the  {\bf co-area formula}, which asserts that when
	 $h$ is Lipschitz and $m:=\dim M\ge m':=\dim M'$  we have for $u\in L^1(M)$:
	\begin{equation}\label{eq_coarea}
		\int_{y \in M'}\left( \int_{h^{-1}(y)}u(x)\, d\hn^{m-m'}(x) \,\right) d y  =\int_{x\in M } u(x)\,\jac h(x)\, d x .
	\end{equation}
%where .
	 %(defined at the points where $h$ is differentiable).
\end{itemize}
	
	\section{The main result}\label{sect_main}
\subsection{Subanalytic sets.}\label{sect_sub}	We refer the reader to  \cite{bm, ds, livre} for all the basic facts about subanalytic geometry. 

	\begin{dfn}\label{dfn_semianalytic}
		A subset $E\subset \R^n$ is called {\bf semi-analytic} if it is {\it locally}
		defined by finitely many real analytic equalities and inequalities. Namely, for each $a \in   \R^n$, there are
		a neighborhood $U$ of $a$ in $\R^n$, and real analytic  functions $f_{ij}, g_{ij}$ on $U$, where $i = 1, \dots, r, j = 1, \dots , s_i$, such that
		\begin{equation}\label{eq_definition_semi}
			E \cap   U = \bigcup _{i=1}^r\bigcap _{j=1} ^{s_i} \{x \in U : g_{ij}(x) > 0 \mbox{ and } f_{ij}(x) = 0\}.
		\end{equation}

		The flaw of the  semi-analytic category is that  it is not preserved by analytic morphisms, even when they are proper. To overcome this problem, we prefer working with the  subanalytic sets.

		A subset $E\subset \R^n$  is  {\bf  subanalytic} if
		each point $x\in\R^n$ has a neighborhood $U$ such that $U\cap E$ is the image under the canonical projection $\pi:\R^n\times\R^k\to\R^n$ of some relatively compact semi-analytic subset of $\R^n\times\R^k$ (where $k$ depends on $x$).

		A subset $Z$ of $\R^n$ is  {\bf globally subanalytic} if $\hh_n(Z)$ is   subanalytic, where $\hh_n : \R^n  \to (-1,1) ^n$ is the homeomorphism defined by $$\hh_n(x_1, \dots, x_n) :=  (\frac{x_1}{\sqrt{1+|x|^2}},\dots, \frac{x_n}{\sqrt{1+|x|^2}} ).$$

		We say that {\bf a mapping $f:A \to B$ is   globally subanalytic}, $A \subset \R^n$, $B\subset \R^m$ globally subanalytic, if its graph is a   globally subanalytic subset of $\R^{n+m}$. For simplicity, globally subanalytic sets and mappings will be referred as {\bf definable} sets and mappings.
We denote by $\s_n$ the set of definable subsets of $\R^n$ and by $\D^+(S)$, $S\in \s_n$, the space of positive continuous definable functions on $S$.%  (this terminology is often used by o-minimal geometers \cite{vdd,cos})In the case $B=\R$, we say that  $f$ is a (resp. globally) {\bf  subanalytic function}.
\end{dfn}

The globally subanalytic category is well adapted to our purpose.  It is stable under intersection, union, complement, and projection. It thus constitutes an o-minimal structure \cite{vdd, cos}, and consequently admits cell decompositions (see \cite[Definition $2.4$]{cos} \cite[Definition $1.2.1$]{livre}) and stratifications (Definition \ref{dfn_stratifications}), from which it follows that definable sets enjoy a large number of finiteness properties (see \cite{cos,livre} for more).

Given $A\in \s_n$, we denote by $A_{reg}$  the set of points of $A$ at which $A$ is a $\cc^\infty$ submanifold of $\R^n$ of dimension $\dim A$. It follows from Tamm's theorem \cite{tamm} that this set is definable, and it easily follows from existence of cell decompositions \cite[Theorem $1.2.3$]{livre} that it  satisfies $\dim A \setminus A_{reg}<\dim A$.   We set  $\pa \omd:=(\delta \omd)_{reg}$, where $\delta \omd :=\adh \omd\setminus\omd$.
% $$\pa \omd :=\{x\in \delta \omd:  \omd \mbox{ is a $\cc^\infty$ of $\R^n$ at $x$}  \}. $$that have a  neighborhood $U$ in $\R^n$ such that each connected component of $U\cap \omd$ is the interior of a  with boundary $U\cap \delta \omd$.
%$\delta \omd$

\subsection{PDEs on subanalytic domains.}\label{sect_PDEs}

We will investigate equation (\ref{eq_L}) under mixed boundary conditions
\begin{equation}\label{eq_boudary_conditions}
	\tra_{\pa \omd} u=0 \mbox{ on } \omd_D,\quad  \gub \nabla u =\theta \mbox{ on } \omd_N,
\end{equation}
where $\omd_D \ne \emptyset$ and $\omd_N$ are two open subanalytic subsets of $\pa \omd$ covering a dense part of this set, and $\theta$ is a function on $\pa \omd$. Here $\tra_{\pa \omd}$ denotes the trace operator on $\pa \omd$ and $\gub \nabla u$ stands for the inner product of $\nabla u$ with a unit normal vector.
 Since there is no regularity assumption on $\omd$ except subanalicity, (\ref{eq_boudary_conditions}) needs some clarification.

 Let us first define the trace operator on $\pa \omd$. At a point of $\pa \omd$, the germ of $\omd$ has one or two connected components which are smooth domains.  Therefore, the restriction of every $u\in W^{1,p}(\omd)$, $p\in [1,\infty)$,  leaves a trace on $\pa \omd$, which is locally $L^p$. Hence, we can define a trace operator $\tra_{\pa \omd}:W^{1,p}(\omd)\to L^p_{loc}(\pa \omd)^2$, which assigns to every $u$ its  traces provided by the respective restrictions of $u$ to the  local connected components of $\omd$ (when $\omd$ is connected near a point, the second component of the trace is by convention $0$, see \cite{trace, lprime} for more).  This mapping indeed depends on the way the  connected components are enumerated, but we assume this choice to be made once for all, as a change of enumeration anyway just induces a permutation of the components of $\tra_{\pa M}$.
 Given an open subset $Z$ of $\pa \omd$, we then set $$W^{1,p}(\omd,Z):=\{u\in W^{1,p}(\omd):\tra_{\pa \omd} u =0 \mbox{ on } Z\}.$$
 It was shown in \cite[Theorem $4.1$]{lprime} that the elements of this space can be approximated by smooth functions (up to the boundary) vanishing in the vicinity of $\adh Z$ for all $p\in [1,2]$. In particular, $\cc_0^\infty(\omd)$ is dense in the kernel of the trace operator for such $p$.  
 It is also worthwhile mentioning that by Proposition $4.5$ of \cite{lprime}, the elements of  $W^{1,p}(\omd,\pa \omd)$ can be extended by $0$ to elements of $W^{1,p}(\R^n)$ for all $p$. We will thus sometimes regard them as defined on $\R^n$.
 %, i.e. $$W^{1,p}(\omd,\pa \omd)=W^{1,p}_0(\omd)$$ 

 The operator $\gub$ can  be defined in a similar way. Let $\wbf_{\!\! \dvv}^{1,p'}(\omd)$ stand for the space of vector fields $\beta\in L^{p'}(\omd)^n$ satisfying $\dvv \beta \in L^{p'}(\omd)$ ($p\in [1,\infty]$, $1/p+1/p'=1$). At a point $x$ of $ \pa \omd$, this set admits a tangent space. Choose a unit normal vector $\nu(x)$, and set for $\beta \in W^{1,\infty}(\omd)^n$, when $\omd$ is connected near $x$, $\gub \beta :=\nu\cdot \beta \in L^\infty(\pa \omd)$. Again, when there are two connected components locally near $x$, one has to do it for the restriction of $u$ to every connected component and one gets a couple of elements of $L^\infty(\pa \omd)$ (again the second component of $\gub$ is zero whenever $\omd$ is connected near $x$, we also need to pay attention to orientation, see \cite{gupel} for more).

 So far, we defined $\gub$ on the space $ W^{1,\infty}(\omd)^n$ and, as in the the case of regular domains, we need a density theorem to extend it to $\wbf_{\!\! \dvv}^{1,p'}(\omd)$.
This was carried out for $p\in (1,2]$ in \cite{gupel}, where the theorem below was established.
 Of course, like the mapping $\tra_{\pa \omd}$, the obtained mapping $\gub:\wbf_{\!\! \dvv}^{1,p'}(\omd)\to W^{-1/p',p'}_{loc}(\pa \omd)$,  where $W^{-1/p',p'}_{loc}(\pa \omd)$ stands for  the dual of the image of the trace operator, depends on the way the two connected are enumerated, and the theorem below requires that coherent choices are made with those we made for $\tra_{\pa \omd}$.

	 \begin{thm}\label{thm_trace_formula}
	 	For every $p\in (1, 2]$,  $\gub$ uniquely extends to a continuous linear operator $\gub:\wbf_{\!\! \dvv}^{1,p'}(\omd)\to W^{-1/p',p'}_{loc}(\pa \omd)$,
	 	such that for all  $\beta\in \wbf_{\!\!\dvv}^{1,p'}(\omd)$ and
	 	$u\in W^{1,p}(\omd)$:
	 	\begin{equation}\label{eq_trace_theorem}
	 		<\beta, \nabla u>+<\dvv\, \beta,u>=<\gab_\nub \beta,\tra_{\pa \omd} u>.
	 	\end{equation}
	 \end{thm}
	 This theorem is derived from a density result of the author  which enables us to approximate elements of $ \wbf_{\!\!\dvv}^{1,p'}(\omd)$, $p\in(1,2]$, by elements of $ W^{1,\infty}(\omd)^n$ vanishing near the singularities \cite[Theorem $3.5$]{gupel}.
 Although these  notions of traces only provide values of the functions and vector fields at smooth points of the boundary, formula (\ref{eq_trace_theorem}) makes it possible to derive the following weak formulation of problem (\ref{eq_L}) with mixed boundary conditions (\ref{eq_boudary_conditions}) when $\theta\in  W^{-1/2,2}_{loc}(\pa \omd)$: for all $\varphi\in \cc^\infty(\adh\omd)$ vanishing in the vicinity of $\delta \omd\setminus \overline{\omd_N}$ we have
\begin{equation}\label{eq_weak_formulation}
 <A\nabla u, \nabla \varphi>=<\theta, \nabla \varphi>-<f,\varphi> .
\end{equation}
Existence and uniqueness for all $f\in L^2(\omd)$ and $\theta\in  W^{-1/2,2}_{loc}(\pa \omd)$ (we assume $\omd_D\ne \emptyset$) then follow from Lax-Milgram's Theorem, together with Poincar\'e inequality for subanalytic domains \cite[Proposition $5.6$]{lprime}.
If $\omd_N=\emptyset$ then existence is guaranteed for all $f$ in $W^{1,2}(\omd,\pa \omd)'=W^{-1,2}(\omd)$ (see \cite[Proposition $4.11$]{lprime}), and we have
 for $u\in W^{1,2}(\omd,\pa \omd)$
\begin{equation}\label{eq_Delta_isom}
 ||u||_{W^{1,2}(\omd)}\lesssim ||L u||_{W^{-1,2}(\omd)}.
\end{equation}
Furthermore, it follows from \cite[Theorem $0.5$]{jerker} and \cite[Theorem $1$]{meyers} that, when the domain is Lipschitz (with $A$ possibly discontinuous), there is $p>2$ such that for  $u\in W^{1,2}(\omd,\pa \omd)$ we have
 \begin{equation}\label{eq_cpctly_supp_meyers}
 ||u||_{W^{1,p}(\omd)}\lesssim ||L u||_{L^{2}(\omd)}.
\end{equation}
 This remains true in any bounded domain $\omd$ (with non necessarily Lipschitz boundary)  if $u\in W^{1,2}(\omd )$ is compactly supported  for we can always replace $\omd$ with a big ball containing it and $u$ with its extension by $0$.
This article   provides a $W^{1,p}$  estimate at points at which Dirichlet's condition is required which is valid on any bounded subanalytic domain (Theorem \ref{main}).% We however underline (and we will make use of it) that

%The constant $C$ thus just depends on the diameter of the domain.When the domain the domain is nonsmooth and nonconvex, this is no longer true (see for instance  \cite[???]{arendt})., although $Lu\in L^2(\omd)$ does not suffice to ensure $u\in W^{2,2}(\omd)$

\subsection{A priori estimates of the solutions.}\label{sect_a_priori}
	 Proving a priori estimates will require
	 assumptions on the tangent cone at singular points of the boundary.  Let us first define the link of the tangent cone (i.e. the section by the unit sphere) at a point $t\in X$ as
	$$\cbt_t(X):=\{v\in \sph^{n-1}:\exists (x_i)_{i\in \N} \to t \mbox{ in } X\setminus \{t\}  \mbox{ with } \frac{x_i-t}{|x_i-t|} \to v\}.$$
	The tangent cone is then defined as the (positive) cone over this set, i.e.: 
	$$\cbf_t(X):=\{\lambda v :\;\; \lambda\ge 0 \et v\in \cbt_t(X) \}.$$
	We will prove a priori estimates   under the assumption that for every $t\in \delta \omd$ 
	\begin{equation}\label{ass_cone_tangent}
		\hn^{n-1} (\cbt_t (\R^n\setminus \omd))\ge \alpha \quad \mbox{ or } \quad \hn^{n-2} (\cbt_t (\delta \omd))\ge \alpha,
	\end{equation}
where $\alpha$ is a positive real number independent of $t$. Namely:

\begin{thm}\label{main}
	Let a bounded subanalytic domain $\omd\subset \R^n$ satisfy (\ref{ass_cone_tangent})
	at each $t\in \delta \omd$ for some $\alpha>0$ independent of $t$. If  $L$ fulfills the ellipticity condition (\ref{eq_ellipticity}) then there are $p>2$ and a constant $C$ such that for every  $u\in W^{1,2}(\omd,\pa \omd)$:
	$$||u||_{W^{1,p}(\omd)} \le C||Lu||_{L^2(\omd)}.$$ 
\end{thm}	

	This theorem makes it possible to provide in many cases a priori estimates under mixed conditions (\ref{eq_boudary_conditions}), as shown by the example below. 
Assumptions on the tangent cone have the advantage to be often easily checked from the equation describing the domain.

\begin{exa}\label{exa1}
Let us consider problem (\ref{eq_argmin}) and set $\omd=\bou_{\R^n}(0,1)\setminus \{P=0\}$.
	 Any solution $u\in W^{1,2}(\bou_{\R^n}(0,1))$ to this problem   satisfies
\begin{equation}\label{eq_pb_lap}\begin{cases}
	\Delta u=0\mbox{ on } \omd\\
	\tra_{\pa \omd} u =g \mbox{ on } \omd_D\\
	\gub \nabla u=0  \mbox{ on } \omd_N,
\end{cases} \end{equation}
 where $\omd_N:=\{x\in \sph^{n-1}: P(x)\ne 0\}$, and $\omd_D:=\pa \omd \setminus \overline{\omd_N}$. After subtracting $g$ to $u$, we get a function $\tilde u$ satisfying
$$\begin{cases}
	\Delta \tilde u=-\Delta g, \quad \mbox{ on } \omd,\\
	\tra_{\pa \omd} \tilde u =0 ,\;\;\quad \mbox{ on } \omd_D,\\
	\gub \nabla \tilde u=-\gub \nabla g, \quad \mbox{ on } \omd_N.
\end{cases} $$
 Existence and uniqueness of the solution to this system are established Lax-Milgram's Theorem and Poincar\'e inequality if $g\in W^{1,2}(\bou_{\R^n}(0,1))$ satisfies $\Delta g\in L^2(\bou_{\R^n}(0,1))$, as explained in the preceding section.  In $\R^3$,  Theorem B of \cite{kai} ensures the continuity of the solution at points of $\delta \omd$ at which Dirichlet condition is required and where $g$ is continuous. Moreover, it directly follows from  Theorem \ref{thm_trace_formula} that any continuous solution of (\ref{eq_pb_lap}) is a solution of (\ref{eq_argmin}) as soon as $\pa \omd$ is dense in  $\delta \omd$.% (i.e. when $\pa \omd$ is dense in $\delta \omd)$.

 Let now $n=3$, $P(x,y,z)=x^3+y^2-z^2x^2$, and $g\in W^{2,2}(\bou_{\R^3}(0,1))$, and
let us check that (\ref{ass_cone_tangent}) holds on $\omd_D$. %Since, we only  need to check this condition for each connected component of $\omd\setminus\{P=0\}$, we will focus on $\{P>0\}$ and $\{P<0\}$  separately. The tangent cone at a point $(0,0,z_0)$ is included in the zero locus of the initial form $y^2-z_0^2x^2$. Moreover,
As $\nabla P=(3x^2-2xz^2,2y,-2x^2z)$, the singular locus of $\{P=0\}$ is reduced to the $z$-axis.  Since  $P(x,y,z)=0$ amounts to $y=\pm \sqrt{x^2z^2-x^3}$,  we have  for   $t=(0,0,z_0)$, $z_0\in \R$,
$$\cbf_{t}(\{P=0\})\supset \{y=\pm z_0x, x< 0\} ,$$
 %Hence, $\cbt_{t}(\{P=0\}) $ contains  $\{(x,y,z)\in\sph^{2}:y=\pm z_0x \}$ centered at $(0,1,0)$,
  which means that  the second inequality of (\ref{ass_cone_tangent}) holds, and therefore, thanks to Theorem \ref{main}, that  $u \in W^{1,p}_{loc}(\bou(0,1))$  for some $p>2$. Here we write $loc$ because we have not addressed the case of Neumann's condition in this article and we can only apply our theorem locally at singular points where Dirichlet's condition holds. However, in the situation of the present example we have $\omd_N=\sph^2\setminus \{P=0\}$ so that, using a classical trick, we can perform a reflection through the sphere to get a solution $u$ with purely Dirichlet type conditions. Hence, $u\in W^{1,p}(\bou(0,1))$ for some $p>2$.%, and   by \cite[Theorem B]{kai} $u$ is continuous.
%Condition (\ref{ass_cone_tangent}) implies that $\{P=0\}$ has pure dimension $2$, which yields that the solution of (\ref{eq_pb_lap}) is the  solution of (\ref{eq_argmin}).
	\end{exa}
	We checked the second inequality of (\ref{ass_cone_tangent}) because the first one fails at the origin for one connected component of $\omd$. The first one is however useful in the case where the domain has an outward horn, like for instance $\omd:=\{x \in \bou_{\R^3}(0,1):x^2+y^2<z^4\}$. The worst case is of course when inward horns arise, because both  may then fail.
We find here relevant to include an example where   this happens.%is included in the zero locus  of the initial form, i.e. $\cbf_0(\pa \omd) \subset \{y^2=0\}$, and  $P(x,y,z)=0$ as well as $y(t)<<x(t)$ forces $z(t)^2/x(t)\to 1$.  Hence, The situation is more delicate at the origin.
\begin{exa}\label{exa2}
Let us take the same situation as in (\ref{eq_pb_lap}), with now $P(x,y,z):=y^2+x^4-z^4x^2$. The singular locus being again equal to the $z$-axis, the same argument as above yields   (\ref{ass_cone_tangent}) for both connected components $\{ P<0 \}$ and $\{ P>0\}$ at every $(0,0,z)$ such that $z\ne 0$, giving, via Theorem \ref{main}, local a priori estimates near those points. It is easily derived from the formula defining $P$ that    $\cbf_0(\pa \omd) $ coincides with the $z$-axis and that (\ref{ass_cone_tangent}) fails for the connected component $\{P>0\}$ at the origin.
 Hence, at the origin, Theorem \ref{main} only yields an a priori estimate for the other connected component.  
 
 That (\ref{ass_cone_tangent}) only fails at the origin can be accounted by the fact that this is the only $0$-dimensional stratum of a $(w)$-regular stratification of the boundary, as (\ref{eq_trivialite_cone_normal}) ensures that (\ref{ass_cone_tangent}) holds at points of strata of codimension smaller than $3$.   Effective methods for computing regular stratifying families of polynomials being available in the semi-algebraic category  \cite[Theorem $9.1.6$]{bcr}, this provides us an effective way to ensure regularity of the solution.

\end{exa}

\section{A few tools from subanalytic geometry}\label{sect_a_few_tools}

%	
%	Given a $\cc^\infty$ submanifold $S$ of $\R^n$, we denote by $\pi_S:U_S\to S$ the local retraction which assigns to each $x$ the point $\pi_S(x)$ that realizes the distance $d(x,S)$, i.e., $d(x,S)=|x-\pi_S(x)|$. This mapping is well-defined and   $\cc^\infty$ on a neighborhood $U_S$ of $S$, and   when $S$ is subanalytic, so is  $\pi_S$ \cite[Proposition]{livre}.
%	

	 \subsection{\L ojasiewicz's inequality.}\label{sect_loj} 	 
	   An important tool of subanalytic geometry is \L ojasiewicz's inequality \cite{loj}, that we recall  in a slightly modified form. We refer to \cite{ania, livre} for a proof.
	 
	 \begin{pro}\label{pro_lojasiewicz_inequality}(\L ojasiewicz's inequality)
	 	Let $f$ and $g$ be two globally subanalytic functions on a  globally subanalytic set $A$ with $\sup\limits_{x\in A} |f(x)|<\infty$. Assume that
	 	$\lim\limits_{t \to 0} f(\gamma(t))=0$ for every
	 	globally subanalytic arc $\gamma:(0,\ep) \to A$ satisfying $\lim\limits_{t \to 0} g(\gamma(t))=0$.
	 		 	Then there exist $\nu \in \N$ and $C \in \R$ such that for any $x \in A$:
	 	$$|f(x)|^\nu \leq C|g(x)|.$$
	 \end{pro}

\subsection{Stratifications and a key  proposition on stratified domains}\label{sect_stratifications}
The {\bf angle} between two given vector subspaces $E$ and $F$ of $\R^n$ will be
estimated as:
$$\angle(E,F):=\underset{u\in E,|u|
	\le 1}{\sup} d(u,F).$$
\begin{dfn}\label{dfn_stratifications}
	A {\bf
		stratification of}\index{stratification} a definable set $X\subset \R^n$ is a finite partition of it into
	definable $\cc^\infty$ submanifolds of $\R^n$, called {\bf strata}\index{stratum}.
% 	A stratification is {\bf compatible} with a set  if this set is the union of some strata. \index{compatible!  stratification}
	%We then say that $(X,\Sigma)$ is a {\bf stratified set}\index{stratified set}.A {\bf refinement of a stratification}\index{refinement! of a stratification} $\Sigma$ of $X$ is a stratification of $X$ compatible with every stratum of $\Sigma$.
%	We say that $\Sigma$ {\bf satisfies the frontier condition} if for every $S\in \Sigma$ the set $fr(S)\cap A$ is the union of some elements of $\Sigma$. 
%	

	 Let $S$ and $S'$ be a couple of strata (of some stratification)   and let $z \in S\cap \adh{S'}$.
	We will say that $(S',S)$ satisfies the {\bf $(w)$ condition at $z$} (of Kuo-Verdier) if there exists a constant $C$ such that for $x\in S'$ and $y
	\in S$ in a neighborhood of $z$:
	\begin{equation}\label{eq_w}
		\angle (T_y S,T_x S') \leq C
		|y-x|.\end{equation}

A stratification $\Sigma$ is said to be {\bf $(w)$-regular} if every couple  of strata of $\Sigma$ satisfies the $(w)$ condition.
We say that a stratification $\Sigma$ of a set $X$ {\bf satisfies the frontier condition} if the closure in $X$ of every  $S\in \Sigma$ is a union of strata of $\Sigma$. Equivalently, given $S$ and $S'$ in $\Sigma$, either $\adh{S'}\cap X\cap S= \emptyset$ or $S\subset \adh {S'}$. In the latter case, if $S'\ne S$ we then write $S\prec S'$.  Since strata are subanalytic, this forces $\dim S<\dim S'$. When strata are connected, the $(w)$ condition entails the  frontier condition (see \cite{ver} or \cite[Propositions $2.6.3$ and $2.6.17$]{livre}).  \end{dfn}
The following proposition, which yields a Poincar\'e type estimate near a stratum of a $(w)$-regular stratification, will be crucial in the proof of Theorem \ref{main}.
\begin{pro}\label{pro_vshort}
	Let $\Sigma$ be a $(w)$-regular stratification of $\delta \omd$ and let $S\in \Sigma$. If (\ref{ass_cone_tangent}) holds at all $t\in S$ for some $\alpha>0$ independent of $t$, then there are $\delta \in \D^+(S)$ and $C>0$ such that for all $u\in W^{1,p}(  \omd,  \pa\omd)$, $p\in [1,+\infty)$, we have for almost every $\eta$ if $\dim S<n-1$:
	\begin{equation}\label{ineq_vshort}
		||u||_{L^p(V^\delta_\eta\cap \omd)} \le C \eta ||\nabla u||_{L^p(V^\delta_\eta\cap \omd)},
	\end{equation}
where $V^\delta_\eta:= \{x\in U^\delta_S:d(x,S)=\eta\}$, and, if $\dim S=n-1$:
\begin{equation}\label{ineq_vshort_reg}
	||u||_{L^p(V^\delta_\eta\cap \omd)} \le C \eta^{1-\frac{1}{p}}\, ||\nabla u||_{L^p(U^{\delta_\eta}_S\cap \omd)},
\end{equation}
where $\delta_\eta$ is the function defined by $\delta_\eta(x):=\min(\delta(x), \eta)$. 
	\end{pro}
The proof of this proposition is postponed to section \ref{sect_pf_vshort}. We first wish to explain why it makes it possible to prove the main theorem. Let us already observe that, by the coarea formula (\ref{eq_coarea}), this proposition yields (assuming (\ref{ass_cone_tangent})) for $u\in W^{1,p}( \omd, \pa \omd)$
\begin{equation}\label{ineq_short_U}
		||u||_{L^p(U^{\delta_\eta}_S\cap \omd)} \lesssim\, \eta \,||\nabla u||_{L^p(U^{\delta_\eta}_S\cap \omd)}.
\end{equation}

\section{Proof of the main result}\label{sect_pf_main}
We are going to show Theorem \ref{main}, admitting Proposition \ref{pro_vshort}.  Let us   fix for all this section a $(w)$-regular stratification $\Sigma$ of $\delta \omd $.  We will assume that strata are connected (this can always be obtained by a refinement), which yields the frontier condition (see Definition \ref{dfn_stratifications}).%, and we will regard the connected components of $\R^n\setminus \delta\omd$ as extra strata of this stratification.
\subsection{The weight functions $\rho_k$}\label{sect_rho_k} An important step of the proof of Theorem \ref{main} consists of establishing weighted estimates near the strata (Lemma \ref{lem_short}), which requires to define weight functions.

\begin{lem}\label{lem_weights}
%Let for each $k$, $X_k:=\bigcup \{S\in \Sigma:\dim S\le k\}$.  For every $0\le k\le n$,$\rho_k\in W^{2,\infty}(\R^n\setminus X_{k}) $ , for all $x\in \R^n $ and such that $|D f|$ and $|D^2 f|$ are both bounded on bounded subsets of $\R^n\setminus X$.
Given a definable set $X\subset \delta \omd$, 
	there is a definable function    $f_X$ that belongs to $W^{2,\infty}(\omd) $ and satisfies for some $\kappa>0$ \begin{equation*}%\label{eq_rho_k_et_dist_N}
			\left(\frac{1}{2}d(x, X )\right)^\kappa \le f_X(x)\le \left(2 d(x, X )\right)^\kappa .
		\end{equation*}%$f^{-1}(0)=X$, $\cc^2$ on $\R^n\setminus X$
	\end{lem}
\begin{proof}
For each $\ep\in \D^+ (\omd)$, we can find a definable $\cc^2$ approximation of $ \R^n\setminus \adh X\ni x\mapsto d(x,X)$,  say $g:\R^n\setminus \adh X\to \R$, satisfying $|g(x) -d(x,X)|<\ep(x)$ for all $x\notin  \adh X$ (see \cite{shi, esc} or \cite[Proposition $2.7.1$]{livre}).  Note that if $\ep(x)\le \frac{d(x,X)}{2}$ then we have $g(x)\le 2 d(x,X)$ as well as $d(x,X)\le 2 g(x)$.
The derivative of $g $ can only tend to infinity at points of $\adh X $. As $g $ vanishes on this set, and since  $\nabla (g ^\kappa )= \kappa g ^{\kappa -1} \nabla g $, it follows from \L ojasiewicz's inequality (applied to $g$ and $\frac{1}{1+|\nabla g |}$, which can only tend to $0$ at points of $\adh X $) that $f_X :=g ^\kappa_{|\omd}\in W^{1,\infty}(\omd)$ if $\kappa $ is chosen  large enough. For the same reason,  $|D_x \nabla f_X|$ is bounded
 if $\kappa$ is sufficiently large.
\end{proof}
 We now define our weight functions by applying this lemma to unions of strata. Namely, we set
 for each $0\le k\le n$
 \begin{equation}\label{eq_weights}
 	\rho_k:=f_{X_k} \,,\quad \mbox{where $X_k:=\bigcup \{S\in \Sigma:\dim S\le k\}$.}  \end{equation} 

\subsection{Mollifying with parameters}\label{sect_mol} The following convolution product will be of service. Let $V$ be a definable submanifold of $\R^n$.
For  $u\in L^p(V\times\R)$ and $\psi\in L^1(\R)$,
we define a function $u \stt \psi$ on $V\times\R $ by setting for almost every $(x,y)\in V\times\R$
\begin{equation}\label{mol}
	u\stt\psi(x,y)=\int_{\R}u(x,z)\, \psi(y-z)\,dz.
\end{equation} We write $\stt$ rather than $*$ to emphasize that it is not the usual convolution product.
It easily follows from Young's inequality and standard arguments (see \cite[Section $2.4$]{trace} or \cite[Section $4.2$]{dukulan} for details) that if  $\varphi:
\R\to \R$ is a nonnegative smooth compactly supported function satisfying $\int \varphi =1$ and if we set $\varphi_\sigma(x):=\frac{1}{\sigma}\varphi(\frac{x}{\sigma})$, then $u\stt\varphi_\sigma$ tends to $u$ in $W^{1,p}(V\times \R)$ as $\sigma\to 0$ for all $u\in W^{1,p}(V\times \R)$.

\subsection{Weighted estimates near strata.} We start with a lemma that yields a weighted inequality near a stratum of $\Sigma$ (we recall that, throughout this section, $\Sigma$ denotes a fixed stratification of $\delta \omd$). 
The proof of Theorem \ref{main} will rely on an induction on the codimension of the strata of $\Sigma$ and we will need this lemma to perform the induction step.

%
%Proposition \ref{pro_vshort}  requires $\dim S<n-1$. For  $(n-1)$-dimensional strata, we will rely on the following similar inequality.

The proofs will involve cut-off functions so as to localize the problem near a stratum of the boundary, and the following functions will be useful to construct them.
 Let $\psi:\R \to [0,1]$  be a decreasing $\cc^\infty$ function
such that $\psi\equiv 1$ on $(-\infty,\frac{1}{2})$ and $\psi\equiv 0$ on $[\frac{3}{4},\infty)$, and set for $\eta$ positive and $s\in \R$: \begin{equation}\label{eq_psieta}\psi_\eta(s):=\psi\left(\frac{s}{\eta}\right),\end{equation}
as well as $\psi _0(s)=0$ for $s\ge 0$ and $1$ for $s< 0$. Observe that we have for $\eta>0$:
\begin{equation}\label{eq_psieta_ineq}
	\sup \psi_\eta=1,\quad  \sup \left|\psi_\eta'\right|\lesssim \frac{1}{\eta}, \et \sup\left|\psi_\eta''\right| \lesssim \frac{1}{\eta^2},
\end{equation}
as well as
\begin{equation}\label{eq_psi_eta_supp}
 \psi_\eta\equiv 1 \mbox{ on } (-\infty,\,\frac{\eta}{2}] \et \psi_\eta\equiv 0 \mbox{ on } [\frac{3\eta }{4},+\infty) .
\end{equation}

 \begin{lem}\label{lem_short}
 Let $S\in \Sigma$  and assume that there is $\alpha>0$ such that  (\ref{ass_cone_tangent}) holds at every $t\in \delta \omd$.   Then there are  positive constants $C$ and $\mu$ as well as  $\delta\in \D^+(S)$  such that for every $u\in W^{1,2}( \omd,\pa \omd)$ supported in $U^\delta _S$ we have
\begin{equation}\label{ineq_short_w}
 ||\rho_{k}^{-\mu}\nabla u||_{L^{2}( \omd)}\le C  ||L u||_{L^2(\omd)},
\end{equation}
where $\rho_{k}$ is given by (\ref{eq_weights}) and $ k=\dim S$.
\end{lem}

\begin{proof} We may of course assume that $Lu$ is $L^2$. Let $\delta\in \D^+(S)$ be as in Proposition \ref{pro_vshort}.
 We first derive from  (\ref{ineq_vshort}) a $W^{1,2}$ a local estimate near $S$: 
% (Claim \ref{cla1}), that we integrate to show another estimate of the same type (Claim \ref{ste_2}), to finally derive (\ref{ineq_short_w}) using \L ojasiewicz's inequality (Claim \ref{ste_3}).
 
% The proof will be divided in three steps.  Let for
%  $\eta >0$ and $\delta\in \D^+(S)$, $$ V^\delta_\eta:=\{x\in U_S^\delta:d(x,S)=\eta\}.$$

\begin{claim}\label{cla1} There are positive constants  $C$ and $\eta_0$ such that for all $u\in W^{1,2}( \omd,\pa \omd)$ supported in $U^\delta _S$ we have for almost every $0<\eta<\eta_0$:
\begin{equation}\label{eq_ste1}
 || \nabla u||_{L^2(U^{\delta_\eta}_S)}^2\le C\eta (||L u||_{L^2(\omd)}^2+||\nabla u||^2_{L^2(V^\delta_\eta)}),
\end{equation}
where $V_\eta^\delta$  and $\delta_\eta$ are like in (\ref{ineq_vshort}) and (\ref{ineq_short_U}) respectively.
%$\delta_\eta$ is the function defined by $\delta_\eta(x):=\min(\delta(x), \eta)$.
\end{claim}

 To prove this claim, let  for  $\zeta$ and $\eta$ nonnegative
 $$\psi_{\eta,\zeta} ^S (x) := \psi_{\eta,\zeta} (d(x,S)), \quad x \in  \omd,\quad \mbox{ where }\quad
 \psi_{\eta,\zeta} (s) := \psi_{\zeta} (s - \eta),$$
with $\psi_\zeta$ as in (\ref{eq_psieta}). We start with the case $\dim S<n-1$. Applying (\ref{eq_trace_theorem}) to $A\nabla u$ and  $\psi^S_{\eta,\zeta}\cdot u$ (with $u\in W^{1,2}(\omd,\pa \omd)$ supported in $ U^{\delta}_S$), we get for $\eta$ and $\zeta$ positive
\begin{equation}\label{eq_ippw}
 -<A \nabla u\, ,\, \psi^S_{\eta,\zeta}\nabla u>=<L u,\psi^S_{\eta,\zeta} u>+< A\nabla u,u \nabla \psi^S_{\eta,\zeta}>.
\end{equation}
As $u,\nabla u$, and $Lu$ are all $L^2$, we have
\begin{equation}\label{eq_psieta_pa_u}\lim_{\zeta\to 0}  |<A \nabla u\, ,\, \psi^S_{\eta,\zeta}\nabla u>|= | <A \nabla u\, ,\,\psi^S_{\eta,0}\nabla u>|\overset{(\ref{eq_ellipticity})}\ge  \lambda_0 \cdot ||\nabla u||_{L^{2}(U^{\delta_\eta}_S)}^2 .\end{equation}
Moreover, since $\psi^S_{\eta,\zeta}\le 1 $ and $\supp\, \psi^S_{\eta,\zeta}\subset U^{\nu}_S$ where $\nu:=\eta+\zeta$,  we see that
\begin{equation*}%\label{ineq_term01}
|<L u, \psi^S_{\eta,\zeta} u>|\le  ||L u||_{L^{2}(\omd)} || u||_{L^{2}(U^{\delta_\nu}_S)}
\overset{(\ref{ineq_short_U})}{\lesssim}   \nu\, ||L u||_{L^{2}(\omd)} ||\nabla u||_{L^{2}(U^{\delta_\nu}_S)}   ,
\end{equation*}
so that, passing to the limit,
\begin{equation}\label{ineq_term1}
\lim_{\zeta\to 0} |<L u, \psi^S_{\eta,\zeta} u>|\lesssim\eta\, ||L u||_{L^{2}(\omd)} ||\nabla u||_{L^{2}(U^{\delta_\eta}_S)}  \overset{(\ref{eq_Delta_isom})}\lesssim     \eta  ||L u||_{L^{2}(\omd)}^2.
\end{equation}
Hence, plugging (\ref{eq_psieta_pa_u}) and (\ref{ineq_term1}) into (\ref{eq_ippw}),  we see that (\ref{eq_ste1}) reduces to show that for almost every $\eta$
\begin{equation}\label{eq_lim_entrl}
\limsup_{\zeta \to 0}| <A\nabla u , u \nabla \psi^S_{\eta,\zeta} >|\lesssim  \eta \, ||\nabla u||_{L^2(V_\eta^\delta)}^2.
\end{equation}
To show (\ref{eq_lim_entrl}), notice that,  we have
\begin{equation*}%\label{eq_link_inner}
 <A\nabla u, u \nabla \psi^S_{\eta,\zeta} >= \int_{U^{\delta}_S}  A \nabla u\cdot  u\nabla\psi^S_{\eta,\zeta} 
\overset{(\ref{eq_coarea})}{=}  \int_0^\infty\int_{V_\theta ^\delta}  A(x) \nabla u(x)  \cdot u(x)\nabla\psi^S_{\eta,\zeta} (x) dx\, d\theta ,
\end{equation*} since $\jac d(\cdot ,S)=|\nabla d(\cdot ,S) |\equiv 1$. As $A$ is $L^\infty$ and $|\nabla \psi^S_{\eta,\zeta} (x)|=-\psi'_{\zeta} (d(x,S)-\eta)$, which is supported in $\{\eta < d(x,S)<\eta+\zeta\}$,   we  derive
\begin{equation}\label{eq_link_inner}
	|<A\nabla u, u \nabla \psi^S_{\eta,\zeta} >|\lesssim- \int_\eta^{\eta+\zeta}\int_{V_\theta ^\delta} |\nabla u(x)|  \cdot| u(x)|\cdot \psi'_{\zeta} (\theta -\eta)\, dx\, d\theta .
\end{equation}
%since $\psi'_{\zeta}$ is supported in $[0,\zeta]$. 

  By   \cite[Theorem $7.6.3$]{paramreg} (or \cite[Corollary $3.2.12$]{livre}) there are $\eta_0>0$ and a definable family of bi-Lipschitz homeomorphisms
$\Phi_\theta :V_{\eta_{_0}}^\delta  \to V_\theta ^\delta , \; \theta \in (0,\eta_0)  $,
 such that $\Phi_\theta (\omd\cap V_{\eta_{_0}}^\delta)=\omd\cap V_\theta ^\delta$. %The Lipschitz constant of $\Phi_s$ being a definable function of $s$, it is continuous for $s\le \eta_0$ if $\eta_0$ is small enough, and hence can be assumed to be bounded independently of  $s\in (\frac{\eta_0}{2},\eta_0) $. % stays bounded away from zero. In particular, for $s\ge $ the restricted mapping $\Phi_s:V^{\eta_0} \to V^s$ is a bi-Lipschitz homeomorphism with a constant that can be bounded  independently of $s$. ince $\psi'_{\eta,\zeta} (s)=\psi'_{\zeta} (s-\eta)$, 
  Setting for simplicity
$$v(x,\theta ):=  |\nabla u(\Phi_\theta (x)) | \cdot u(\Phi_\theta (x)) \cdot\jac \Phi_\theta (x),
$$
and pulling-back by means of $\Phi_\theta $ the integral on $V^\delta_\theta $ that sits in the right-hand-side of (\ref{eq_link_inner}), we get (applying also Fubini) for $\eta\in (0,\eta_0)$ and $\zeta\in (0,\eta_0-\eta)$: %&\le& \int_\zeta^\eta\int_{V^\delta _\eta} \psi'_{\eta,\zeta} (s)v(x,s)dx\, ds
\begin{equation}\label{eq_pr_lim_inner}
|  <A\nabla u, u \nabla \psi^S_{\eta,\zeta} >|
   \lesssim -\int_{V^\delta _{\eta_{_0}}} \int_\eta^{\eta+\zeta}v(x,\theta )\, \psi'_{ \zeta} (\theta -\eta)dx\, d\theta =||v \stt \tilde\psi_\zeta' ||_{L^1(V^\delta_{\eta_{_0}}\times\{\eta\})},
%    &\le &\int_{V^\delta _\eta}
\end{equation}
%which is equal to $   || v \stt \psi_\zeta' ||_{L^1(V^\delta _\eta)}$
where $\stt$ stands for the convolution product introduced in (\ref{mol}) and $\tilde\psi_\zeta'$ is the derivative of the function $\tilde\psi_\zeta(s):=\psi_\zeta(-s)$. The right-hand-side of (\ref{eq_pr_lim_inner}) tends to (for almost every $\eta$, see the paragraph right after (\ref{mol}) for explanations)
\begin{equation}\label{eq_pr_lim_inner2}|| v  ||_{L^1(V^\delta_{\eta_{_0}} \times\{\eta\})}= ||u\nabla u  ||_{L^1(V^\delta _\eta)} \le  ||u  ||_{L^2(V^\delta _\eta)} ||\nabla u  ||_{L^2(V^\delta _\eta)}\overset{(\ref{ineq_vshort})}{\lesssim} \eta  ||\nabla u||_{L^2(V^\delta _\eta)}^2,\end{equation}
 establishing (\ref{eq_ste1}) in the case $\dim S<n-1$.
 
  In the case $\dim S=n-1$, we will replace (\ref{eq_lim_entrl}) with 
  \begin{equation}\label{eq_lim_entrl_reg}
  	\limsup_{\zeta \to 0}| <A\nabla u , u \nabla \psi^S_{\eta,\zeta} >|\lesssim  \eta^\frac{1}{2}  \,  ||\nabla u||_{L^2(U^{\delta_\eta}_S)} ||\nabla u||_{L^2(V_\eta^\delta)} ,
  \end{equation}
replacing in its proof (\ref{ineq_vshort}) with (\ref{ineq_vshort_reg}). Namely,    (\ref{eq_pr_lim_inner2}) can be replaced with 
\begin{equation*}|| v  ||_{L^1(V^\delta_{\eta_{_0}} \times\{\eta\})}= ||u\nabla u  ||_{L^1(V^\delta _\eta)} \le  ||u  ||_{L^2(V^\delta _\eta)} ||\nabla u  ||_{L^2(V^\delta _\eta)}\overset{(\ref{ineq_vshort_reg})}{\lesssim} \eta^\frac{1}{2} ||\nabla u||_{L^2(U_S^{\delta _\eta})}||\nabla u||_{L^2(V^\delta _\eta)},\end{equation*}%\label{eq_pr_lim_inner2_reg}
establishing (\ref{eq_lim_entrl_reg}). Plugging 
 (\ref{eq_psieta_pa_u}), (\ref{eq_lim_entrl_reg}), and  the first inequality of (\ref{ineq_term1})  into (\ref{eq_ippw}), we get, after dividing by  $ || \nabla u||_{L^2(U^{\delta_\eta}_S)} $:
 $$  || \nabla u||_{L^2(U^{\delta_\eta}_S)} \lesssim\eta  ||L u||_{L^2(\omd)} +\eta^\frac{1}{2}||\nabla u||_{L^2(V^\delta_\eta)},$$
 which yields Claim \ref{cla1}.
% $$v= |u| |\nabla u|\circ \Phi \cdot \jac \Phi$$
%%%%%%%%%%%%%%%%%%%%%%%%%%%%%%%%%%%%%%%%%%%%%%%%%%%%%%%%%55

\begin{claim}\label{ste_2}
There is  $\sigma>0$ such that for  $u\in W^{1,2}( \omd,\pa \omd)$ supported in $U^\delta_S$ and $0<\eta<1$ we have:
 \begin{equation}\label{ineq_short}
 ||\nabla u||_{L^2(U^{\delta_\eta}_S)}\lesssim \eta^{\sigma} ||L u||_{L^2(\omd)},
\end{equation}
where $\delta_\eta$ is as in (\ref{ineq_short_U}).
% again $\delta_\eta$ is the function defined by $\delta_\eta(x):=\min (\delta(x), \eta)$.
\end{claim}
Set for simplicity
\begin{equation*}
\lambda(\eta):=||\nabla u||_{L^2(U^{\delta_\eta}_S)}^2 =\int_0^\eta  ||\nabla u||^2_{L^2(V^\delta_\zeta)}d\zeta,
\end{equation*}
   and notice that then $\lambda'(\eta)=||\nabla u||_{L^2(V^\delta _\eta)}^2$ as distributions for almost every $\eta$, so that  (\ref{eq_ste1})  reads for $\eta<\eta_0$
%if we plug the two inequality just above into (\ref{eq_ste1}), we get ( \begin{equation}
%  || \nabla u||_{L^2(U^\eta)}^2\le   ||L u||_{L^2(\omd)}^2+ |<u  ,\nabla u>_{V^\delta _\eta}|.
% \end{equation}
% Since this majorant is independent of $i$, we get after passing to the limit in (\ref{eq_ippw}), thanks to (\ref{eq_lim_entrl})
%|| u||_{W^{1,2}(U^\eta)}
\begin{equation*}\lambda(\eta)\le C \eta   ||L u||^2_{L^2(\Omega)}+C\eta \lambda'(\eta).
\end{equation*}
 We may assume $C>2$, and hence $c:=\frac{1}{C}<\frac{1}{2}$. Dividing the above inequality by $C\eta^{c+1}$, we obtain
\begin{equation*}c\,\frac{\lambda(\eta)}{\eta^{c+1}}-\frac{ \lambda'(\eta)}{\eta^c}\le \eta^{-c} ||L u||^2_{L^2(\Omega)}.
\end{equation*}
% where we setted for simplicy $C_3=\frac{C_1}{C_2}   ||L u||^2_{L^2(\Omega)}$.
Integrating on $[\eta,\eta_0]$, we get for almost every $0<\eta< \eta_0$
\begin{equation*}\frac{ \lambda(\eta)}{\eta^c}\le  \frac{\lambda(\eta_0)}{\eta_0^c} +\frac{\eta_0^{1-c}}{1-c}\, ||L u||^2_{L^2(\Omega)},
\end{equation*}
which yields Claim \ref{ste_2} for $\eta<\eta_0$ with $\sigma=\frac{c}{2}$ (since $\lambda(\eta_0) \overset{(\ref{eq_Delta_isom})}{\lesssim}  ||L u||^2_{L^2(\omd)} $). For $\eta\in[\eta_0,1]$, the desired fact follows from (\ref{eq_Delta_isom}).
% \begin{equation}
%  ||\nabla u||_{L^2(U^\eta)}\lesssim  \eta^\frac{c}{2} ||L u||_{L^2(\omd)}.
% \end{equation}
% For $\eta\ge 1$, the desired estimate comes from (\ref{eq_Delta_isom}), completing step \ref{ste_2}.

%%%%%%%%%%%%%%%%%%%%%%%%%%%%%%%%%%%%%%%%%%%%%%%%%%%%%%%%%%%%%%

%\begin{claim}\label{ste_3}
%There  is $N>0$ such that for  $u\in W^{1,2}(\omd,\pa \omd)$ we have
% \begin{equation*}
% || \rho_{k}^{-\frac{\sigma}{2N}}\nabla u||_{L^2(U^\delta_S)}  \le C  ||L u ||_{L^2(\omd)},
% \end{equation*}
%where $\sigma$ and $\delta$ are provided by Claim \ref{ste_2}.
% \end{claim}

% By Lemma \ref{lem_short}, the stratum $S$ has a neighborhood $U$ such that inequality (\ref{ineq_short}) holds for all $u\in W^{1,2} (U\cap \Omega, U\cap \pa \Omega)$.  Taking $U_S$ smaller if necessary, we thus can assume that this estimate holds for $U=U_S$.We may also assume that $\rho_k\le 1$ on $U_S$.

We are now ready to establish (\ref{ineq_short_w}).
Note that thanks to the frontier condition, taking $\delta\in \D^+(S)$ smaller if necessary (making $\delta$ smaller does not affect (\ref{ineq_short}), which is what we need), we can assume $S$ to be the only stratum of dimension $\le k$  meeting $U^\delta_S$. Hence, for such $\delta$, if $\gamma:(0,a) \to\omd\cap U^\delta_S$ is a subanalytic arc such that $\lim_{s\to 0}\rho_k(\gamma(s))= 0$  then  $\gamma(s)$ goes to a point of $\adh S$ as $s\to 0$, which means that $d(\gamma(s),S)$ tends to zero. Thus, for such $\delta$, by \L ojasiewicz's inequality, there are $C$ and $N$ such that    for $x\in \omd\cap U^\delta_S$ \begin{equation}\label{ineq_loj_rho_k_dist}d(x,S)^{N}\le C\rho_k(x).\end{equation}
% We claim that for $u\in W^{1,2}(U_S)$ we havewhere $\alpha$ is given by (\ref{ineq_short}).
We will also assume $\delta<1$, so that if for $i\in \N$ we set $W_i:=\{x\in U^\delta_S\cap \omd: 2^{-(i+1)}<d(x,S)<2^{-i}\} $ and if $\sigma$ is as in (\ref{ineq_short}), we have
\begin{equation*}
 || \rho_{k}^{-\frac{\sigma}{2N}}\nabla u||^2_{L^2( U^\delta_S\cap \omd)}= \sum_{i\in \N}   ||\rho_{k}^{-\frac{\sigma}{2N}}\nabla u||_{L^2(W_i)}^2
 \overset{(\ref{ineq_loj_rho_k_dist})}{\lesssim}  \sum_{i\in \N}   ||d(\cdot,S)^{-\frac{\sigma}{2}}\nabla u ||_{L^2(W_i)}^2 \end{equation*}
 \begin{equation*}
 \le \sum_{i\in \N}   ||\nabla u ||_{L^2(W_i)}^2 2^{(i+1)\sigma}\overset{(\ref{ineq_short})}\lesssim\sum_{i\in \N}   ||L u ||_{L^2( U^\delta_S\cap \omd)}^2   2^{(1-i)\sigma} \lesssim  ||L u ||_{L^2( U^\delta_S\cap \omd)}^2   ,
\end{equation*}
%for some $C'$ that just depends on $C$ and $\sigma$,
establishing (\ref{ineq_short_w}) with $\mu:=\frac{\sigma}{2N}$.
\end{proof}

%%%%%%%%%%%%%%%%%%%%%%%%%%%%%%%%%%%%%%%%%%%%%%%%%%%%%%%%%%%%%%%%%%%%%%

\subsection{Proof of Theorem \ref{main}.}We are now ready to show the main result of this article, admitting Proposition \ref{pro_vshort}. For this purpose, we are going to prove
by decreasing induction on $k$, that for every integer $-1\le k\le n-1$ and each $N\in \N$ sufficiently large there are  $C>0$ and $p>2$ such that for all  $u\in W^{1,2} (\omd,\pa \omd)$  we have
\begin{equation}\label{eq_west_k}
||\rho_k^N\nabla u||_{L^{p}(\omd)}\le C||Lu||_{L^2(\omd)},
\end{equation}
 with $\rho_k$ given by (\ref{eq_weights}) and by convention $\rho_{_{-1}}\equiv 1$. The statement of the theorem   follows from the case $k=-1$ and Poincar\'e inequality.

\begin{step}\label{step_n}
 We prove (\ref{eq_west_k}) in  the case $k=n-1$. \end{step}
% \begin{proof}[Proof of step \ref{step_n}.]We first check that for $N>2 $,    the first two derivatives of $\lambda_\eta$ are bounded by a constant independent of $\eta$.   Indeed, b
% Take $N\in \N$ (we shall establish the desired estimate  for $N$ large enough). 
 Define a cut-off function by setting for $x\in \omd$, $\lambda_\eta (x):=\rho_{n-1}^{2} (x)(1-
 \psi_\eta (\rho_{n-1}(x)))$, where $\psi_\eta$ is as in (\ref{eq_psieta}). By  (\ref{eq_psi_eta_supp}), we have $\supp\,\nabla (\psi_\eta\circ \rho_{n-1})\subset \{x\in \omd :\rho_{n-1} (x)<\eta\},$
 which, thanks to the estimate of the derivatives of $\psi_\eta$ given in (\ref{eq_psieta_ineq}), means that $\lambda_\eta'$ and $\lambda_\eta''$ are bounded  independently of $\eta$.
 Since $A\in W^{1,\infty}(\omd)^{n\times n}$, we have on $\omd$ for $u\in W^{1,2}(\omd,\pa \omd)$
 \begin{eqnarray}\label{eq_Deltau_wn}
| L (\lambda_\eta u)|\lesssim |\lambda_\eta L u| +| \nabla \lambda_\eta |\cdot (|u|+|\nabla u|) +  |u \cdot L \lambda_\eta  |, \end{eqnarray}
which, by (\ref{eq_Delta_isom}), yields  \begin{equation}\label{eq_L_phi_rho_u}
              ||L(\lambda_\eta u)||_{L^2(\omd)} \lesssim ||L u||_{L^2(\omd)}.
             \end{equation}
As $\lambda_\eta u$ is compactly supported in $\omd$ and coincides with $\rho_{n-1}^2 u$ on some open set $U_\eta$ containing $\rho_{n-1} \ge \eta$, we can write
\begin{equation}\label{eq_rho2u}
||\rho_{n-1}^2 u||_{W^{1,p}(U_\eta)} \le  ||\lambda_\eta u||_{W^{1,p}(\Omega)} \overset{(\ref{eq_cpctly_supp_meyers})}{\lesssim}    ||L (\lambda_\eta u)||_{L^2(\omd)}\overset{(\ref{eq_L_phi_rho_u})}{\lesssim}  ||L  u||_{L^2(\Omega)},
\end{equation}
with a constant independent of $\eta$. As $\rho_{n-1}\in W^{2,\infty}(\omd)$ and $\rho^3 _{n-1}\nabla u=\rho _{n-1}\nabla (\rho^2 _{n-1}u)-2\rho^2 _{n-1}u\nabla  \rho_{n-1}$, we see that
 \begin{equation*}\label{eq_1jet_wghted}
 ||\rho^3 _{n-1}\nabla u||_{L^{p}(\Omega)} \lesssim  ||\rho_{n-1}^{2} u||_{W^{1,p}(\Omega)} .
\end{equation*}
 Since $\bigcup U_\eta=\omd$, (\ref{eq_rho2u}) therefore entails that
\begin{equation}
 ||\rho^3 _{n-1} \nabla u||_{L^{p}(\Omega)} \lesssim  ||L  u||_{L^2(\Omega)},\end{equation}
which completes the first step of the induction, with $N=3$.%\end{proof}

%%%%%%%%%%%%%%%%%%%%%%%%%%%%%%%%%%%%%%%%%%%%%%%%%%%%%%%%%%%%%%%%%%%%%%

Let us now fix $k<n$, assume (\ref{eq_west_k}) to be true for this value of $k$, and take a $k$-dimensional stratum $S$.
We first establish the desired estimate on $U_S^\delta$ for some  $\delta\in \D^+(S)$ small. More precisely:
%Given a continuous definable function $\delta:S\to (0,+\infty)$, we define the tubular neighborhood of  radius $\delta$ as
% $$U_S^\delta:=\{x \in U_S :|x-\pi_S(x)|<\delta(\pi_S(x)) \}.$$
% and that $\rho_{k+1}=\min\{ d(x,Y) :Y\in \Sigma, \dim Y\le k\}$ is equal to $d(x,S)$ on $U_S$.

\begin{step}\label{step_ineq}
We  show that there is  $\delta_S \in \D^+(S)$ such that for every $N\in \N$ sufficiently large there are $C>0$ and $q>2$ such that for all $u\in W^{1,2}(\omd,\pa \omd)$
\begin{equation}\label{eq_rec_udelta}
 ||\rho_{k-1}^N\nabla u||_{L^{q}(U^{\delta_S} _S\cap \omd)}\le C||L u||_{L^2(\omd)}.
\end{equation}
\end{step}
% \begin{proof}[Proof of step \ref{step_ineq}.](we will show the desired statement for $\delta$ smaller than some elements of $\D^+(S)$).
The idea of the proof of (\ref{eq_rec_udelta}) is to replace (\ref{eq_cpctly_supp_meyers}), used in the proof of step \ref{step_n}, with the induction hypothesis. Take $\delta\in \D^+(S)$. We will assume  $\delta<d(\cdot,\R^n\setminus U_S)$ ($U_S$ being the domain of our local retraction $\pi_S$), which entails that the topological frontier of $U^\delta_S$ in $U_S$  coincides with $\{x\in U_S:d(x,S)=\delta (\pi_S(x))\}$.
Moreover,  we will assume $\delta$ to be smaller than $d(\cdot ,\adh S\setminus S)$ and to the one provided by Lemma \ref{lem_short}. %$U^\delta_ S$ does not meet a stratum of dimension  $< k$ (thanks to the frontier condition) and .

Because the inequality given by this lemma only holds for functions that are supported in $U^\delta_S$, we need to introduce some cut-off functions. Applying \L ojasiewicz's inequality to the functions $\delta\circ \pi_S$ and  $\rho_{k-1|U_S^\delta}$ (the required condition for this inequality to hold is satisfied if $\delta <d(\cdot ,\adh S\setminus S)$ thanks to the frontier condition), we   see that there are $d\in \N$ and  a constant $C$ such that for all $x\in U^\delta _S$ \begin{equation}\label{eq_Loj_tub}
\rho_{k-1} ^{d}(x)\le C \delta (\pi_S(x)).                                                                                                                                                                                                     \end{equation}
 Define then a smooth function by setting for $x\in U_S^\delta$
 $$\psi_{S}(x):=\psi\left(\frac{2Cd(x,S)}{\rho_{k-1}^d(x)}\right) ,$$
 where $\psi$ is as in (\ref{eq_psieta}).
  Because (\ref{eq_Loj_tub}) forces $\frac{2Cd(x,S)}{\rho_{k-1}^{d}(x)}\ge 1$ for all $x \in U^\delta _S\setminus U^\frac{\delta}{2}_S$ and since $\psi$ is supported in $[0,\frac{3}{4}]$,  $\psi_{S}$ may be smoothly extended by $0$ on $\omd\setminus U_S^\delta$ (at frontier points of $U_S^\delta$, we have $ d(x,S)=\delta (\pi_S(x))$, see the paragraph right after (\ref{eq_rec_udelta})).
 Set   for $N\in \N$,
 $$u_{S,N}:=\rho_{k-1}^N\cdot \psi_{S}\cdot u,$$ and notice that,
since (by a straightforward computation of derivative) $|\nabla\psi_{S,N}| \lesssim \rho_{k-1}^{-d}$ and $|D\nabla\psi_{S,N}| \lesssim  \rho_{k-1}^{-2d}$, similarly as in (\ref{eq_Deltau_wn}) we have for $N$ large enough
\begin{equation}\label{eq_Deltau_wn_k}
 |L u_{S,N}| \lesssim  (|u|+|\nabla u|+ |L u|).
\end{equation}

Let $p$ stand for the number provided by the induction hypothesis (see (\ref{eq_west_k})), and set $\nu:=\frac{\mu}{2}$, where $\mu>0$ is the number appearing in (\ref{ineq_short_w}). Applying H\"older's inequality to the functions $|\nabla u_{S,N}|^\theta \rho_k^\nu$ and $ |\nabla u_{S,N}|^{1-\theta} \rho_k^{-\nu}$, with $q\in (2,p)$ and $\theta\in (0,1)$ such that  $\frac{\theta}{p}+\frac{1-\theta}{2}=\frac{1}{q}$, we get
\begin{eqnarray}\label{ineq_nabla_usn_Lq}
 ||\nabla u_{S,N}||_{L^q(U_S^\delta)} &\le& \left(\int_{U_S^\delta} |\nabla u_{S,N}|^{p} \rho_k^\frac{\nu p}{\theta} \right)^\frac{\theta}{p}\cdot \left(\int_{U_S^\delta} |\nabla u_{S,N}|^{2} \rho_k^{-\frac{2\nu}{1-\theta}} \right)^\frac{1-\theta}{2}\nonumber\\
 &=& ||\rho_k^\frac{\nu}{\theta}\nabla u_{S,N} ||_{L^{p}(U_S^\delta)}^\theta\cdot || \rho_k^{-\frac{\nu}{1-\theta}}\nabla u_{S,N} ||_{L^2(U_S^\delta)}^{1-\theta}.
\end{eqnarray}
For $\theta$ small enough (i.e. for $q>2$ close enough to $2$), applying the induction assumption to $u_{S,N}$ we see that  \begin{equation}\label{eq_west_usN}
||\rho_k ^{\nu/\theta}\nabla u_{S,N}||_{L^{p}(U_S^\delta)}\lesssim ||L u_{S,N}||_{L^2(\omd)}.
\end{equation}
Moreover, for $\theta\le \frac{1}{2}$ (which entails $\frac{\nu}{1-\theta}\le 2\nu= \mu$), applying (\ref{ineq_short_w}) to $u_{S,N}$  we get
\begin{equation}\label{eq_usN_short}
||\rho_k^{-\nu/(1-\theta)} \nabla u_{S,N} ||_{L^2(U_S^\delta)}\lesssim||L u_{S,N}||_{L^2(\omd)}.
\end{equation}
Plugging (\ref{eq_west_usN}) and  (\ref{eq_usN_short}) into (\ref{ineq_nabla_usn_Lq}) we get
\begin{eqnarray*}
 ||\nabla u_{S,N}||_{L^q(U_S^\delta)}\lesssim ||L u_{S,N}||_{L^2(\omd)}\overset{(\ref{eq_Deltau_wn_k})}{\lesssim } ||L u||_{L^2(\omd)}.
 \end{eqnarray*}
 Since $\psi_{S}$ is equal to $1$ on  some definable neighborhood of $S$, this   completes step \ref{step_ineq}.%\end{proof}

 %%%%%%%%%%%%%%%%%%%%%%%%%%%%%%%%%%%%%%%%%%%%%%%%%%55
 \begin{step}
  We now perform the induction step for our fixed value of $k$.%, completing the proof of (\ref{eq_west_k}).
 \end{step}

 Step \ref{step_ineq} provides  for each $k$-dimensional stratum  $S$ a function $\delta_S\in \D^+(S)$ for which  (\ref{eq_rec_udelta}) holds. Furthermore, it follows from \L ojasiewicz's inequality that there are  $c\in \N$ and $C>0$ such that on  $W:=\omd\setminus \bigcup \{  U_S^{\delta_S}:S\in \Sigma, \dim S=k\}$ we have
 \begin{equation}\label{eq_loj_rho}\rho^{c} _{k-1} \le C \rho_k .\end{equation}
 Hence, if $p$ is the number given by the induction assumption, we can write for $N$ sufficiently large
and $u\in W^{1,2}(\omd,\pa \omd) $
$$||\rho_{k-1}^{c N}\nabla u||_{L^{p}(W)}\overset{(\ref{eq_loj_rho})}\lesssim ||\rho_k^N\nabla u||_{L^{p}(W)} \overset{(\ref{eq_west_k})}\lesssim ||Lu||_{L^2(\omd)},$$
which, together with  (\ref{eq_rec_udelta}) (for each $k$-dimensional stratum $S$ of $\Sigma$), completes our induction step.

\section{Equisingularity of the tangent cone of $(w)$-stratified sets}\label{sect_equi}
This section presents some results about Kuo-Verdier stratifications that will be useful to establish (\ref{ineq_vshort}) in the next section. The main purpose is to establish formulas  (\ref{eq_lim_f_t}) and (\ref{eq_trivialite_cone_normal}), which are closely related to the so-called normal pseudoflatness of $(w)$-regular stratifications \cite{hir, orrtro}. 
We fix for all this section a $(w)$-regular stratification $\Sigma$ of a closed definable subset $X$ of $\R^n$. We will assume that strata are connected.

% Before entering the technical details, let us present a few basic facts. 

%Let $X$ be stratified by a $(w)$-regular stratification $\Sigma$ and let $S\in \Sigma$. 
% Since our observations are local, we will assume $S=\R^k \times \{0_{\R^{n-k}}$
\subsection{Rugose vector fields}\label{sect_rugose} Let $S\prec S'$ be a couple of  strata of $\Sigma$ and let $\xo\in S$.  It directly follows from the definition of the $(w)$ condition (more details are given below) that  each smooth tangent vector field  on $S$, say $v:S\to TS$, has an extension $\tilde v:S'\cup S\to \R^n$, tangent to $S'$ on $S'$ and  satisfying for some constant $C$, for all $x\in S'$ and $y \in S$ in the vicinity of $\xo$:
 \begin{equation}\label{eq_rugose}
 |\tilde v(x)-\tilde v(y)|\le C |x-y|.\end{equation}
    Vector fields satisfying such an inequality are called {\bf rugose} \cite{ver}. This is a Lipschitz type condition but $y$ has to be in $S$ while $x\in S'$. We here stress the fact that the constant depends on $\xo$. %More generally a rugose vector field $w$ on a stratified set $(X,\Sigma)$ is a mapping $w:X\to \R^n$ such that $w(x)\in T_x S$, for every $x\in S\in \Sigma$ and such that (\ref{eq_rugose}) for every couple of strata $S\prec S'$ for $y\in S$ and $x\in S'$ close to $y$.
 
 The  $(w)$ condition thus ensures that we can extend smooth tangent vector fields on a stratum to rugose tangent vector fields on a neighborhood of a point. This vector field can be defined by $\tilde v(x):=P_x(\hat v(x))$, where $\hat v$ is a smooth vector field (not necessarily tangent to strata)  extending $v$ to $S\cup S'$   and $P_x$ the orthogonal projection onto  $T_xS'$.  When there are more than two strata, one can extend a given vector field to higher dimensional strata by induction on the dimension of the strata. We refer to \cite{ver} or \cite{orrtro} for more, and indeed  give some more details in the proof of Proposition \ref{pro_ppn} below, which relies on these techniques.   The difference is however that we will require in addition  the rugose vector field to be orthogonal to $\nabla d(\cdot,S)$ in order to obtain integral curves that preserve the distance to $S$. This extra fact, which will be useful to study equisingularity of sections by spheres  along strata (see (\ref{eq_d_h_link})), is our motivation for proving Lemma \ref{lem_angles}.
 
%  $P_x:\R^n \to  T_x S'$ denotes the orthogonal projection and if  the vector $ w(x):=P_x(w(\pi_Y(x))$, $x\in X$ close to $\xo$, satisfies
% $$|w(x)-w(y)|\le C|x-y|$$  

\subsection{Continuity of the normal cone along strata.}\label{sect_normal_cone}  Given $S\in \Sigma$,   we set for $t\in S$
\begin{equation}\label{eq_cone_normal}
\cbf^S_t (X):=\cbf_t ( \pi_S^{-1}(t)\cap X) \et  \cbt^S_t (X):=\cbt_t ( \pi_S^{-1}(t)\cap X),
\end{equation}
and call $\cbf^S_t (X)$ the {\bf normal cone} of $X$ along $S$ at $t$.
% Hence, if we set
% \begin{equation}\label{eq_cone_normal}
% 	\cbf_x^S(X) :=\cbf_x (X\cap T_x S^\perp)\et	\tilde{\bf C}_x^S(X) :=\cbt_x (X\cap T_x S^\perp),
% \end{equation}
%and if $x*E$ denotes the cone over a set $E\subset S^{n-1}$ at $x$,
%we then have $$\cbf^S_x(X):=x*\cbt_x^S(X).$$

Before entering our study of the variation of the normal cone along strata, let us recall a well-known property of the tangent cone. Given $t\in X$ and $r>0$, set ($\lbf$ stands for ``link'')
\begin{equation*}
\lbf_{r,t}(X):=\{\frac{x-t}{r}: |x-t|=r, x\in X \}\subset \sph^{n-1}.
\end{equation*}
It is a well-known fact (justified just below)  that subanalytic germs cannot oscillate in the radial direction, in the sense that for each $t\in X$
\begin{equation}\label{eq_lim_f_fibre}
	\lim_{r\to 0}	d_\hn(\lbf_{r,t}(X),\cbt_t(X))=0.
\end{equation}
 We shall show  that the $(w)$ condition is sufficient to ensure that (\ref{eq_lim_f_fibre}) is also true about the normal cone (see (\ref{eq_lim_f_t})).
This issue was studied in
\cite{hir, orrtro}, and we will rely on the same method.

\begin{proof}[Proof of (\ref{eq_lim_f_fibre})]
 Let $r_i>0$ be a sequence tending to zero and let $w_i\in \cbt_t(X)$ be such that  $d_\hn(\lbf_{r_i,t}(X),\cbt_t(X))= d(w_i,\lbf_{r_i,t}(X))$. Extracting a sequence if necessary, we can assume that   $w_i$  tends to some  $w$. Since $w\in \cbt_t(X)$, it follows from Curve Selection Lemma \cite[Proposition $2.2.4$]{livre} or \cite[Theorem $3.2$]{cos}, that there is a definable arc $\gamma:(0,\ep)\to X$ tending  to $t$ at $0$ such that $\lim_{s\to 0} \frac{\gamma(s)-t}{|\gamma(s)-t|}=w$. We can parameterize this arc by its distance to $t$, so that $\frac{\gamma(r)-t}{|\gamma(r)-t|}\in \lbf_{r,t}(X)$ for all $r>0$ small, which entails that
$$d_\hn(\lbf_{r_i,t}(X),\cbt_t(X))= d(w_i,\lbf_{r_i,t}(X))\le |w_i-w|+ |w- \frac{\gamma(r_i)-t}{|\gamma(r_i)-t|}|, $$
which tends to zero.%|w_i-\frac{x_i-t}{r_i}|$y_i=\frac{x_i-t}{r_i}$ for some $x_i$ and,
\end{proof}

 We will need a lemma about angles between transverse sections of vector spaces with a hyperplane.

\begin{lem}\label{lem_angles}
For any pair of vector subspaces $P$ and $Q$ of $\R^n$ we have for all $v\in \sph^{n-1}\setminus Q^\perp$
$$\angle (P\cap v^\perp, Q\cap v^\perp)\le 2\;\frac{\angle(P,Q)}{d(v,Q^\perp) } .$$
\end{lem}
\begin{proof}We denote by $\pi_V$ the orthogonal projection onto a space $V$. Take $v\in \sph^{n-1}\setminus Q^\perp$ and
a unit vector $u$ in $P\cap v^\perp$. We have
\begin{eqnarray}\label{eq_u_moins_piqu}
|u-\pi_{Q\cap v^\perp}(u)| &\le& |u-\pi_{Q}(u)| +|\pi_Q(u)-\pi_{Q\cap v^\perp}(u)|
\nonumber\\
&\le& \angle(P,Q)+ |<\frac{\pi_Q(v)}{|\pi_Q(v)|},\pi_{Q}(u)>|,
\end{eqnarray}
for $\pi_Q(u)-\pi_{Q\cap v^\perp}(u)$ is the projection of $\pi_Q(u)$ onto the orthogonal complement of $Q\cap v^\perp$ in $Q$, which is the line generated by $ \frac{\pi_Q(v)}{|\pi_Q(v)|}$. Observe now that
 $$ |<\frac{\pi_Q(v)}{|\pi_Q(v)|},\pi_{Q}(u)>|=|<\frac{ v}{|\pi_Q(v)|},\pi_{Q}(u)>|=\frac{ |<v,u-\pi_{Q}(u)>|}{|\pi_Q(v)|} \le \frac{\angle(P,Q)}{d(v,Q^\perp)}, $$
 for $|u-\pi_{Q}(u)|\le \angle(P,Q)$ and $|\pi_Q(v)|=d(v,Q^\perp)$. Plugging the just above inequality into (\ref{eq_u_moins_piqu})  yields the desired fact.
% stands for a unit vector of $Q$ normal to $Q\cap v^\perp$.
\end{proof}

The following proposition will establish stability of the normal link  (see (\ref{eq_d_h_link})) and  triviality of the normal cone  along strata (see (\ref{eq_trivialite_cone_normal})).
%expresses that for $S\in \Sigma$ germs of fibers of the restriction of $\pi_S$ to our stratified space $X$ are homeomorphic with a homeomorphism that stays close to the identity when fibers are close to each other. This equisingularity result 

\begin{pro}\label{pro_ppn}Given $t_0\in S\in \Sigma$, there   are a number  $\ep>0$, a  neighborhood $W$ of $t_0$ in $S$, and a homeomorphism $$h:(\pi_S^{-1}(t_0)\cap U_S^\ep)\times W\to \pi_S^{-1}(W)\cap U_S^\ep$$
satisfying
\begin{enumerate}[(i)]
	  \item $\pi_S( h(x,t))=t$ and  $d(h(x,t),S)=d(x,S)$, for all $(x,t)$.
	  % with
	%$h:\pi_S^{-1}(W)\cap \{d(\cdot,S)<\ep\}\to \pi_S^{-1}(t_0)\cap \{d(\cdot,S) <\ep\}$ satisfies
	
 \item\label{item_approx_id} There is  $C$ such that for all $x\in \pi_S^{-1}(t_0)\cap U_S^\ep$ we have $|h(x,t)-x|\le Cd(x,S)\cdot |t|$, for all $t$  in $W$.

 \item\label{item_H_preserve_les_strates} $h$ preserves the strata, i.e. $h(x,t)\in Y$ for all $t\in W$ if $x\in Y\cap \pi_S^{-1}(t_0)\cap U_S^\ep$, $Y \in \Sigma$.
\end{enumerate}
\end{pro}
\begin{proof}
The proof relies on classical arguments for proving normal pseudo-flatness \cite{hir,orrtro}. %Since $X$ is closed, we can take the connected components of $\R^n\setminus X$ as extra strata and assume $X=\R^n$.Let $\pi_S:U_S\to S$ be a local closest point retraction.
 For simplicity, we will assume $t_0=\orn$. Let $g:W\to \R^k, g=(g_1,\dots,g_k)$ be a local coordinate system of $S$ at $0$, where $k=
 \dim S$.  We will regard the connected components of the complement of $X$ as extra strata.
 
 We start by constructing a local basis  of rugose vector fields  $v_i:U_S^\ep \cap \pi^{-1}_S(W)\to \R^n$, $i=1,\dots,k$, $\ep>0$ small (in the sense that $v_1(t),\dots,v_k(t)$ is a basis of $T_tS$ for every $t\in W$), tangent to the strata and smooth on strata, and satisfying $\nabla d(\cdot,S)\cdot v_i\equiv 0$ on $U_S^\ep \cap \pi^{-1}_S(W)\setminus S$ for all $i$. We set $U_S^{\ep,l}:=\bigcup \{Y\cap U_S^\ep :\dim Y\le l, Y\in \Sigma \}$   and we will define the $v_i$'s on  $U^{\ep,l}_{S}\cap \pi^{-1}_S(W)$   by induction on $l$. For $\ep$ small enough, $U_S^\ep$ only meets $S$ and strata $Y\succ S$. We thus can start our induction by $v_i:= \nabla g_i$ on $U_S^{\ep,k}\cap W =W$.

 Assume that the $v_i$'s have been defined on $ U^{\ep,l}_{S}\cap \pi^{-1}_S(W)$ and let us extend them to  $U^{\ep,l+1}_{S}\cap \pi^{-1}_S(W)$. Kuo-Verdier's $(w)$ condition implies Whitney's $(b)$ condition (for subanalytic stratifications, see \cite{kuo} or \cite[Proposition $2.6.3$]{livre}), which means that for each stratum  $Y\succ S$,  $d(\frac{x-
 	\pi_S(x)}{|x-
 	\pi_S(x)|}, T_xY)$ tends to zero as $x\in Y$ tends to a point of $S$. As a matter of fact,  Lemma \ref{lem_angles} and the $(w)$ condition imply that for every $Y\succeq S$, each $y_0\in Y$ has a neighborhood such that   for  $y\in Y$ and  $ x \in Y'\succ Y$ in this  neighborhood
 \begin{equation}\label{eq_w_cap_H}
 	\angle (T_yY\cap (y-
 	\pi_S(y))^\perp,T_xY'\cap (x-
 	\pi_S(x))^\perp)\lesssim |x-y|.\end{equation}  
Now, if $Y\succ S$ is a stratum of dimension not greater than $l$ and if $x\in U^{\ep,l+1}_{S}\setminus  U^{\ep,l}_{S}$ is close to $Y$, we set $\tilde v_{i,Y}(x):=P_x(\hat{v}_i(x))$, where   $\hat{v}_i$ is a  smooth extension of $v_{i|Y}$  to a neighborhood of $Y$ (not necessarily tangent to the strata) and $P_x$ is the orthogonal projection onto $(x-
\pi_S(x))^\perp\cap T_x Y'$, $Y'$ being the $(l+1)$-dimensional stratum containing the point $x$. 
  Inequality (\ref{eq_w_cap_H})  yields the rugosity of $\tilde v_{i,Y}$ for each such stratum $Y$ (we recall that the constant of the rugosity condition is local), that we can glue together using a partition of unity to complete the induction step.

 We shall now produce a basis having the same properties as the $v_i$'s as well as extra properties. Namely, we are going to construct vector fields  $w_i: \pi_S^{-1}(W)\cap U_S^\ep  \to \R^n  $, $i=1,\dots,k$, $\ep>0$ small, tangent to the strata, smooth on strata, and satisfying for all $i$
% \begin{equation}\label{eq_rho_controlled}  \nablad(\cdot,S)\cdot w_i\equiv 0 \et D\pi_S(w_i(x))=w_i(\pi_S(x)).\end{equation}
\begin{enumerate}
\item\label{item_rho_controlled} $\nabla d(\cdot,S)\cdot w_i\equiv 0$.
\item\label{item_pi_controlled} $D_x\pi_S(w_i(x))=w_i(\pi_S(x))$, for all $x\in \pi_S^{-1}(W)\cap U_S^\ep $.
\item\label{item_e_i} $w_{i}(t)\cdot \nabla g_i(t)=\delta_{ij}$ (Kronecker symbol)  for all $t\in W$.
\end{enumerate}
The $w_i$'s will indeed will be obtained as  linear combinations of the $v_i$'s (condition (\ref{item_rho_controlled}) is preserved by linear combinations).
 %To define them notice that, the linear map $\Lambda_x:T_{\pi_S(x)} S\to T_{\pi_S(x)} S$ that sends the $D_x\pi_S(v_i(x))$'s onto the  $v_i(\pi_S(x))$'s must satisfy for some constant $C$ \begin{equation}\label{eq_A_rugose}|\Lambda_x-Id_{T_{\pi_S(x)} S}|\le C|x-\pi_S(x)|. \end{equation}the coefficients of the matrix of $\Lambda_x$ with respect to
 Namely, since  $x\mapsto D_x\pi_S(v_j(x))$ is rugose for each $j$,
 if we set $w_i(x):=\sum_{i=1}^k \lambda_{ij}(x)v_j(x)$, where the $\lambda_{ij}(x)$'s are the
 coordinates of $v_i(\pi_S(x))$ in the basis $D_x\pi_S (v_1(x)),\dots,D_x\pi_S (v_k( x))$, i.e. $v_i(\pi_S(x))=\sum_{i=1}^k \lambda_{ij}(x)D_x\pi_S (v_j(x))$, we get $w_1,\dots,w_k$ satisfying (\ref{item_rho_controlled}) and (\ref{item_pi_controlled}). As  $\nabla g_1(t),\dots, \nabla g_k(t)$ form a basis of $T_t S$ for all $t\in W$, a similar argument applies to get a basis vector fields satisfying   (\ref{item_e_i}) in addition.

To define the desired homeomorphism, let now $\phi_i(x,t)$ denote the local flow generated by the vector field $w_i$, for each $i\le k$. By Gr\"onwall's Lemma,  the rugosity of the $w_i$'s implies  that for $x\in S'\succ S$ and $t\in S$ close to $0$, and $s$ positive small:
\begin{equation}\label{eq_gronwall}
  |x-\pi_S(x)| \exp (-Cs )\le |\phi_i (x,s)-\phi_i (\pi_S(x),s)|\leq    |x-\pi_S(x)| \exp (Cs ).
\end{equation}
The first inequality  establishes that an integral curve starting at $x\in S'$ may not fall into $S$. The second inequality implies that it stays in a little neighborhood of $S$ if $x$ is chosen sufficiently close to $S$. It also yields that $\phi_i $ is continuous ($\phi_i$ being smooth on strata, its continuity just needs to be checked when $x$ passes from one stratum to another one, see \cite{ver} for more). We finally set for $x\in \pi_S^{-1}(0)\cap U_S^\ep$ and $t\in  W$, 
$$h(x,t):=\phi_1(\phi_2(\dots\phi_k(x,t_k)\dots)t_2)t_1),$$
where $(t_1,\dots, t_k)=g(t)$.
 Condition (\ref{item_rho_controlled}) yields $d (\phi_i(x,t),S)\equiv d(x,S)$ for all $i$, which yields in turn $d(h(x,t),S)=d(x,S)$. Conditions (\ref{item_pi_controlled}) and (\ref{item_e_i}) entail that $\pi_S(h(x,t))=t$ for all $x$ and $t$. Property (\ref{item_approx_id}) follows from  the rugosity of the $w_i$'s, which implies $|w_i(x)-w_i(\pi_S(x))|\lesssim d(x,S)$.
\end{proof}

We are now ready to generalize (\ref{eq_lim_f_fibre}) to the normal cone.
Let us define the {\bf normal link along $S\in \Sigma$} (of radius $r>0$) at $t\in S$  as
\begin{equation*}
	\lbf_{r,t}^S(X):=\{\frac{x-t}{r}: |x-t|=r, x\in \pi_S^{-1}(t)\cap X  \}\subset \sph^{n-1}.
\end{equation*}
% Let now $X\in \s_n$ be a closed set stratified by a stratification $\Sigma$ and take $S\in \Sigma$. %as in the above proposition.
% and  set for $t\in S$ and $r>0$
% $$A_{r,t}:=\frac{1}{r}\cdot X\cap\pi_S^{-1}(t)\cap  \sph(0,r).$$
 Because the mapping $x\mapsto h(x,t)$ provided by the above proposition preserves the distance to $S$ for every $t$,  it induces a homeomorphism from $\lbf^S_{r,t_0}(X)$ onto $\lbf^S_{r,t}(X)$ (for $r$ small), so that  (\ref{item_approx_id}) of the latter proposition yields for  $r>0$ smaller than some $r_0$ independent of $t\in S$ close to $t_0$
\begin{equation}\label{eq_d_h_link}
	d_\hn (\lbf^S_{r,t}(X),\lbf_{r,t_0}^S(X))\lesssim |t-t_0|.
\end{equation}
Together with (\ref{eq_lim_f_fibre}), this entails that for all $t_0\in S$ we have
\begin{equation}\label{eq_lim_f_t}\lim_{(r,t)\to (0,t_0)} d_\hn (\lbf^S_{r,t}(X),\cbt^S_{t_0}(X))=0.\end{equation}
It is worth here stressing the difference with  (\ref{eq_lim_f_fibre}) because it was the motivation for proving Proposition \ref{pro_ppn}:    (\ref{eq_lim_f_fibre}) only entails that for each $t_0$  \begin{equation*}\lim_{r\to 0} d_\hn (\lbf^S_{r,t_0}(X),\cbt^S_{t_0}(X))=0.\end{equation*}
Note also that  (\ref{item_approx_id}) of Proposition \ref{pro_ppn} also yields for $t\in S$
\begin{equation}\label{eq_trivialite_cone_normal}
 \cbf_t(X)=\cbf^S_t(X)\times T_tS.
\end{equation}

%%%%%%%%%%%%%%%%%%%%%%%%%%%%%%%%%%%%%%%%%%%%%%%%%%%%%%%%%%%%%%%

\section{Proof of Proposition \ref{pro_vshort}}\label{sect_pf_vshort}
We first wish to recall the definition of definable families because it will be very important for our purpose.  Given $A\in \s_{m+n}$ and $\tim$, we define the {\bf fiber of $A$ at $t$} as
\begin{equation*}%\label{eq_fibre}
	A_t:=\{x\in \R^n:(t,x)\in A\}.
\end{equation*}
Any family $(A_t)_\tim$ constructed in this way is called {\bf a definable family of subsets of $\R^n$}. Hence, a definable family of sets is not only a family of definable sets: the fibers  must glue together into a definable set. This fact will be essential in the proofs of Lemma \ref{lem_contient_boule} and Proposition \ref{pro_vshort} to get constants that are independent of  parameters.
\begin{rem}\label{rem_fibres_cell_dec}
	Let $A \in \s_{m+n}$ and let $\C$ be a cell decomposition of $\R^{m+n}$ compatible with $A$. For $t\in \R^m$, let then $\C_t:=\{C_t:C\in \C\}$. It follows from the definition of  cell decompositions that for every $\tim$, $\C_t$ is a cell decomposition of $\R^n$ compatible with $A_t$.  
\end{rem}

\begin{lem}\label{lem_contient_boule}
	Let  $A$ and $B$ in $\s_{m+n}$ with $\sup_{x,y\in B_t,\tim} |x-y|<\infty$.
	%$E\in \s_m$ satisfy $A_t\subset \bou(0,\beta)$ for all $t\in E$. %Assume that there is $\alpha>0$ such that $\hn^n(A_t)\ge \alpha$ for all $t\in E$.$\hn^n(A_t)\ge a$.
	%  More generally, if $(B_t)_{\tim}$ is another definable family of sets there is $\sigma>0$ such that  $A_t$ contains a ball of radius $\sigma$  for all $t\in E$
	For every  $b>0$  there is $\sigma>0$  such that $A_t$ contains a ball of radius $\sigma$  for all $\tim$   satisfying $ d_{\hn}(A_t,B_t)<\sigma$ and $\hn^n(B_t)\ge b$.
\end{lem}
\begin{proof} We argue by induction on $n$, the case $n=0$ being vacuous. Fix $b>0$.	
	By \cite[Proposition $1.15$]{kp} or \cite[Theorem $3.1.18$]{livre}, there  is a finite partition $\mathcal{P}$ of $\R^{m+n}$ into definable sets which is compatible with $A$ and $B$ (i.e. these sets are unions of elements of $\mathcal{P}$)  and such that for every $\tim$ and  $C\in \mathcal{P}$ the set $C_t$ is  up to a linear isometry (that we will identify with the identity), a Lipschitz cell (see \cite[Definition $3.1.16$]{livre}) whose Lipschitz constants can be bounded independently of $t$.
	% For $t\in E$, we have \begin{equation}\label{ineq_image_pi_N}
		%\hn^{n-1}(\pi_n(B_t))\ge \frac{\hn^n (B_t)}{2b},
		%\end{equation}
		%where .
		
		Clearly, there must be for every $\tim$  an element $C$ of  $\mathcal{P}$ (depending on $t$) comprised in $B$ and such that  $\hn^n (C_t)\ge \frac{\hn^n (B_t)}{N}$, where $N$ is the number of elements of  $\mathcal{P}$. Since $C_t$ is for each $\tim$  (up to a linear isometry) an $n$-dimensional Lipschitz cell, it is delimited by the respective graphs of two Lipschitz functions $\xi_{1,t}$ and $\xi_{2,t}$ defined on an $(n-1)$-dimensional cell $D_t$ of $\R^{n-1}$ and satisfying $\xi_{1,t}<\xi_{2,t}$.
		Let
		\begin{equation}\label{eq_alpha_1}
			b':=\frac{b}{2N\cdot  \hn^{n-1}( \bou_{\R^{n-1}}(0,a))}, \quad a:=\sup_{x,y\in B_t,\tim} |x-y|,
		\end{equation}
		and, denoting by $\pi_n:\R^n \to \R^{n-1}$ the canonical projection,
		$$C'_t:=\{x\in C_t:  (\xi_{2,t}-\xi_{1,t})(\pi_n(x))\ge b'\} \et C''_t:=C_t\setminus C'_t.$$     
		It follows from the definition of $C'_t$ and $C_t''$  that   for $\tim$ we have 
		$$\hn^n(C_t)\le b' \hn^{n-1}(\pi_n(C''_t))+a \hn^{n-1}(\pi_n(C'_t)), $$
		and therefore, for all $t$ such that $  \hn^n(B_t)\ge b $
		(which  implies $\hn^n (C_t)\ge \frac{b}{N}$) we have
		\begin{equation}\label{ineq_hn_Cprime_t}
			\hn^{n-1}(\pi_n(C'_t))\ge \frac{1}{a}\left(\frac{b}{N}-b' \hn^{n-1}(\pi_n(C''_t))\right)\overset{(\ref{eq_alpha_1})}\ge \frac{b}{2Na} .\end{equation} By induction on $n$, $\pi_n(C'_t)$ contains a ball of radius $\sigma$, for some $\sigma$ independent of $t$, for all $t$ such that $  \hn^n(B_t)\ge b $. Since $(\xi_{2,t}-\xi_{1,t})\ge b'$ on $\pi_n(C'_t)$ and because the Lipschitz constants of $\xi_{1,t}$ and $\xi_{2,t}$ can be bounded independently of $t$, we deduce that $C_t'$ (and hence $B_t$) contains  a ball of radius $\sigma':=\min (\sigma, \frac{b'}{M} )$, with $M$ independent of $t$.
		% We claim that $\sigma':=\min (\sigma, \frac{b'}{2M}, a )$ has the required property.
		% Notice first that since $\sigma'\le a$, if $\tim$ is such that $  \hn^n(B_t)\ge b$ and  $A_t\subset \bou_{\R^{n }}(0,a) $, as well as $d_\hn(A_t,B_t)<\sigma'$  then $t\in E$. Moreover,
		%The condition 
		Now, if $d_\hn(A_t,B_t)<\sigma'$ then this ball   must meet $A_t$, which, since $\mathcal{P}$ is also compatible with $A$, entails that $C$ entirely fits in $A$, which means that $A_t$ contains a ball of radius $\sigma'$. % for each $t$ satisfying $\hn^n(A_t)\ge a$.
	\end{proof}
	\begin{rem}\label{rem_lem_boule_dans_sph}
		The above lemma assumes $A_t\subset \R^n$  and finds a ball $\bou_{\R^n}(x_0(t),\sigma)$. But if $A_t\subset \sph^n$ then we may apply this lemma to find a ball $\bou_{\sph^n}(x_0(t),\sigma)$, for we can cover $\sph^n$ by two charts $h_1:U_1\to \bou_{\R^n}(0,1)$ and $h_2:U_2\to \bou_{\R^n}(0,1)$ and apply the lemma to $h_1(A_t)$ and $h_2(A_t)$, which are subsets of $\R^n$.
	\end{rem}

\begin{lem}\label{lem_boules}
	For every $u\in W^{1,p}(\sph^n)$, $n\ge 1$, $p\in [1,\infty)$, and $\xo$ as well as $x_1$ in $\sph^n$ we have for $\sigma>0$
	\begin{equation}\label{claim_balls}
		||u||_{L^p(\bou_{\sph^n}(x_0,\sigma))}\le ||u||_{L^p(\bou_{\sph^n}(x_1,\sigma))}+4||\nabla u||_{L^p(\sph^n)}	.
	\end{equation}
	\end{lem}
	\begin{proof}
		Let $P$ denote the plane generated by $\xo$ and $x_1$ and let $\theta\in (0,\pi]$ be the angle between these two vectors (we choose an orientation of $P$). We denote by $h_t:\sph^n\to \sph^n$ the mapping induced by the rotation with angle $t$ in $P$, leaving invariant the vectors that are normal to $P$. Hence, $h_\theta(\xo)=x_1$ and $h_\theta(\bou_{\sph^n}(x_0,\sigma))=\bou_{\sph^n}(x_1,\sigma)$. Since $h_t$ is an isometry for every $t$, it now suffices to write
		\begin{eqnarray*}
			||u\circ h_\theta-u||_{L^p(\bou(x_0,\sigma))}&=&\left(\int_{\bou(x_0,\sigma)}|u\circ h_\theta-u|^p\right)^{1/p}=\left(\int_{\bou(x_0,\sigma)}\left| \int_0^\theta\frac{d (u\circ h_t)}{dt}\right|^p\right)^{1/p}\\
			&\le& \int_0^\theta \left(\int_{\bou(x_0,\sigma)}|\nabla u \circ h_t|^p\right)^{1/p} \le \theta ||\nabla u||_{L^p(\sph^n)},
		\end{eqnarray*}
		yielding (\ref{claim_balls}).
\end{proof}

% We stress the fact that the just obtained constant $\sigma$, as well as the constant provided by the lemma below, is independent of the parameter $t$ appearing in the conclusion.

\begin{lem}\label{lem_h_t}Given a definable $\cc^2$ submanifold $S$ of $\R^n$ there are $\delta\in \D^+(S)$ and  a family of $\cc^1$ diffeomorphisms $h_t:\pi_S^{-1}(t)\cap U_S^\delta \to T_tS^\perp \cap \bou_{\R^n}(0,\delta(t))  $, $t\in S$,  $\cc^1$ with respect to $t$ and satisfying $D_th_t =Id_{T_tS^\perp}$ as well as $|h_t(x)|=|x-t|$, for all $x\in \pi_S^{-1}(t)\cap U_S^\delta$.
	\end{lem}
\begin{proof}
	Fix $t\in S$.
Let us denote by $P_t:\pi_S^{-1} (t) \to T_t S^\perp \times \{t\}$ the restriction to the manifold $\pi_S^{-1} (t) $ of the (affine) orthogonal projection. Notice that then $D_t P_t=Id_{T_t S^\perp}$, and set for $x\in \pi_S^{-1}(t)$ close to $t$
$$h_t(x):=\frac{|x-t|}{|P_t(x)-t|}(P_t(x)-t). $$
A straightforward computation of derivative shows that we also have $D_t h_t=Id_{T_t S^\perp}$, which means that $h_t$ is a local diffeomorphism near $t$. Existence of a definable function $\delta$ then comes from  Definable Choice \cite[Proposition $2.3.1$]{livre}. \end{proof}
\begin{proof}[Proof of (\ref{ineq_vshort})]
			Let us set for $t\in S$ and $r>0$:
	$$V_{r,t}:= \sph(t,r)\cap \pi_S^{-1}(t)\cap \omd .
	$$
We first show that there are a constant $C$ and a function $\delta\in \D^+(S)$ such that  for $u\in W^{1,p}(\omd,  \pa \omd )$, $t\in S$, and $r< \delta(t)$, we have:%at each  $t\in S$ for some $\alpha$ independent of $t$, 	that there is $\delta>0$ such thatwe have for all $t\in S$ such that  $\hn^{n-k-1}(B_{r,t}')\ge \alpha_0$, and all $r< \delta(t)$
	\begin{equation}\label{ineq_vtr0}
				||u||_{L^p(V_{r,t})}\le C r||\nabla u||_{L^p(V_{r,t})}.
	\end{equation} 
The desired statement will ensue after integration of this inequality (see (\ref{ineq_vtr2})).

To simplify notations, and because this will be  the dimension of the ambient space, we set $\nen:=n-k-1$, $k:=\dim S$.
	Assumption (\ref{ass_cone_tangent}) consisting of the disjunction of two inequalities, given $t\in S$, we naturally distinguish two cases for proving the just above inequality.
	%Indeed,
	%if $t$ is such that %$\hn^{n-1}(\cbt_t(\omd))\ge \alpha$ then,
	
\noindent \underline{\it First case:}  we will focus on the values of $t\in S$ for which \begin{equation}\label{eq_first}
\hn^{n-1}(\sph^{n-1}\setminus \cbt_t(\omd))\ge \alpha.                                                                                     \end{equation}
 Notice that for such $t$, since $\cbf_t(\omd) $ is a cone over $\cbt_t(\omd)$, we must have  $\hn ^{n}(\cbf_t(\omd)\cap \bou(0,1))\ge \alpha/n$, by the coarea formula. But then  by (\ref{eq_trivialite_cone_normal}) we get that
\begin{equation}\label{eq_rond_carre}
	\hn^{n-k} (\cbf_t^S(\omd)\cap \bou(0,1))\ge \frac{\alpha}{ n	\hn^k (\bou(0,1))},\end{equation}
which implies (again, since $\cbf^S_t(\omd) $ is a cone over $\cbt^S_t(\omd)$) that 
\begin{equation}	\label{eq_hyp_S}
 \hn^{\nen} (\cbt_t^S(\omd))\ge\frac{\alpha}{ n	\hn^k (\bou(0,1))}.\end{equation}
Let $h_t:\pi_S^{-1}(t)\cap U_S^\delta \to T_tS^\perp \cap \bou_{\R^n}(0,\delta(t))  $,  $t\in S$, be the homeomorphism given by Lemma \ref{lem_h_t}  and set for $t\in S$ and $r<\delta(t)$
\begin{equation}\label{eq_WAB} W_{r,t}=\sph(t,r)\cap \pi_S^{-1}(t)\setminus \omd, \quad A_{r,t}=\frac{1}{r}\cdot h_t ( W_{r,t}), \quad \mbox{as well as}\quad B_t=\cbt_t^S (\R^n\setminus \omd).\end{equation}
Since we can work up to a definable family of isometries, we will regard $ A_{r,t}$ and $ B_t$, which  are  subsets of $\sph^{n-1}\cap T_t S^\perp$, as  subsets of $\sph^\nen$.
  
 We wish to apply Lemma \ref{lem_contient_boule} to find a ball in $A_{r,t}$ for each $t\in S$ satisfying (\ref{eq_first}) and $r>0$ small, which requires to evaluate $d_\hn(A_{r,t},B_t)$. Let $\sigma$ be given by applying this lemma to the families $A_{r,t}$ and $B_t$, $t\in S$, $r<\delta(t)$, with  $b=\frac{\alpha}{ 	n\hn^k (\bou(0,1))}$ (we here regard $B_t$ as parameterized by $r$, constant with respect to $r$). Since $D_t h_t=Id_{T_tS^\perp}$ and $h_t(t)=0$ (which means that $\Phi_t(x):=h_t(x)-(x-t)$ satisfies $\Phi_t(t)=0$, $D_t \Phi_{t|T_tS^\perp}=0$ and therefore $|\Phi_t(x)|=o(|x-t|)$), there is $\delta\in \D^+(S)$ such that for all $t\in S$ and $x\in \pi^{-1}_S(t)\cap \bou(t,\delta(t))$
 $$|h_t(x)-x+t|=|\Phi_t(x)|\le \frac{\sigma |x-t|}{2},$$
 which entails that for $r< \delta(t)$, we have $d_\hn(A_{r,t},\lbf^S_{r,t}(\R^n\setminus \omd))<\frac{\sigma}{2}$ (since $W_{r,t}$ is merely $\lbf^S_{r,t}(\R^n\setminus \omd)$  shifted by $t$), and consequently:
 \begin{equation}\label{ineq_A_t_B_t}
  d_\hn (A_{r,t},B_t)\le d_\hn(A_{r,t},\lbf^S_{r,t}(\R^n\setminus \omd)) +d_\hn (\lbf^S_{r,t}(\R^n\setminus \omd),B_t)<\frac{\sigma}{2}+d_\hn (\lbf^S_{r,t}(\R^n\setminus \omd),B_t),
 \end{equation}
which, by (\ref{eq_lim_f_t}), implies that if $\delta\in \D^+(S)$ is chosen small enough then $d_\hn(A_{r,t},B_t)<\sigma$ for all $t\in S$ and  $r< \delta(t)$.

 By Lemma \ref{lem_contient_boule} (see Remark \ref{rem_lem_boule_dans_sph})  and (\ref{eq_hyp_S}), $A_{r,t}$ thus must contain a ball of radius $\sigma$ in $\sph^\nen$ for all $t$ satisfying (\ref{eq_first}) (which entails (\ref{eq_hyp_S})) and  $r< \delta(t)$. Lemma \ref{lem_boules} then implies that  for $v\in W^{1,p}(\sph^\nen)$ vanishing on $A_{r,t}$, $t$ satisfying (\ref{eq_first}), and $r< \delta(t)$, we have for all $\xo\in \sph^\nen$
   \begin{equation}\label{eq_sphere_loc1}
  	||v||_{L^p(\bou_{\sph^\nen}(x_0,\sigma))}\le 4||\nabla v||_{L^p(\sph^\nen)},
  \end{equation}
which clearly yields
  \begin{equation}\label{ineq_short_sph}
  	||v||_{L^p(\sph^\nen)}\lesssim ||\nabla v||_{L^p(\sph^\nen)}.
  \end{equation}
 	 Setting  $\tilde{u}_{r,t}(x):=u(r\cdot h_t^{-1}(x))$, $x\in \sph^\nen$ (we recall that $u$ is implicitly extended by $0$), we then can write for  $t$ satisfying (\ref{eq_first}) and $r< \delta(t)$, $\delta \in \D^+(S)$ small enough,
	 	\begin{equation}\label{ineq_vtr}
		||u||_{L^p(V_{r,t})} \lesssim r^\nen  ||\tilde{u}_{r,t}||_{L^p(\sph^\nen)} \overset{(\ref{ineq_short_sph})}{\lesssim}
		r^\nen||\nabla \tilde{u}_{r,t}||_{L^p(\sph^\nen)} \lesssim r||\nabla u||_{L^p(V_{r,t})},
	\end{equation}
	yielding  (\ref{ineq_vtr0}).

	%%%%%%%%%%%%%%%%%%%%%%%%%%%%%%%%%%%%%%%%%%%%%%%%%%%%%%%%5

 \noindent \underline{\it Second case:} we now are going to deal with the values of $t\in S$ for which \begin{equation}\label{eq_second}
 \hn^{n-2}(\cbt_t(\delta \omd))\ge \alpha.                                                                                     \end{equation} 	 Let $\delta(t)$ and $h_t$ be as in Lemma \ref{lem_h_t}, and set for $t\in S$ and $r<\delta(t)$ (compare with (\ref{eq_WAB}))
 $$ W_{r,t}'=\sph(t,r)\cap \pi_S^{-1}(t)\cap \delta \omd, \quad A_{r,t}'=\frac{1}{r}\cdot h_t (W'_{r,t}),\quad \mbox{as well as} \quad B'_t=\cbt^S_t( \delta \omd).$$
 Like  $A_{r,t}$ and $B_t$,  $A_{r,t}'$ and $B_t'$  will be regarded as  families of subsets of $\sph^\nen$.  Our strategy will be to first prove (\ref{ineq_short_sph}) for $v\in  W^{1,p}(\sph^\nen, A_{r,t,reg}')$, $t\in S$ satisfying (\ref{eq_second}) and $r<\delta(t)$  (see (\ref{eq_sphere_loc}) below) and conclude in the same way via (\ref{ineq_vtr}).
 Note that thanks to Lemma \ref{lem_boules}, it suffices to prove this estimate for each such $r,t$ at some $\xo\in \sph^\nen$ (possibly depending on $r,t$) with a constant independent of $r,t$.
 Since we can require $\sigma<1$, we can use a coordinate system of the sphere.  We will therefore from now regard the families $A_{r,t}'$ and $B_t'$ as families of subsets of $\R^\nen$ and we will show that
 there are a positive real number $\sigma$ and a function $\delta\in \D^+(S)$  such that  for each $t\in S$ satisfying (\ref{eq_second}) and $r\in (0,\delta(t))$ there is $\xo\in \R^\nen$ such that for $v\in W^{1,p}(\bou_{\R^\nen}(0,\sigma),A'_{r,t,reg})$
 \begin{equation}\label{eq_sphere_loc} ||v||_{L^p(\bou_{\R^\nen}(\xo,\sigma))} \lesssim ||\nabla v||_{L^p(\bou_{\R^\nen}(\xo,\sigma))}, \end{equation}
 where the constant will be bounded independently of such $r,t$.

  To prove (\ref{eq_sphere_loc}), notice that  the same argument word for word as in (\ref{eq_rond_carre}) and (\ref{eq_hyp_S}) (applied to $\delta \omd$ instead of $\R^n\setminus\omd$) yields (\ref{eq_hyp_S})  for $\delta \omd$, i.e. for each $t\in S$ satisfying (\ref{eq_second})
 \begin{equation}	\label{eq_hyp_S2}
 	\hn^{\nen-1} (\cbt_t^S(\delta \omd))\ge\frac{\alpha}{ n	\hn^k (\bou(0,1))}.\end{equation}
 By Theorem $1.3.2$ of \cite{paramreg}, there is a definable family of bi-Lipschitz homeomorphisms $H_{r,t}:\R^\nen \to \R^\nen, t\in S, r<\delta(t)$, such that $H_{r,t}(A'_{r,t})$ and $H_{r,t}(B_t')$ can, for each $t\in S$ and $r<\delta(t)$, be included in the union of finitely many graphs of definable Lipschitz functions $\xi_{i,r,t} :\R^{\nen-1} \to \R$, $i=1,\dots, l$. Moreover, the Lipschitz constants of these functions as well as the Lipschitz constants of $H_{r,t}$ and $H_{r,t}^{-1}$ can be bounded independently of $r,t$.  Since we can work up to a bi-Lipschitz map for each $r,t$, we will identify $A'_{r,t}$ and $B_t'$ with their images, assuming that they are contained in the graphs of the $\xi_{i,r,t}$'s.

 Let $\C$ be a cell decomposition of $ \R^{1+n+\nen}$ compatible with $$ \{(r,t,x)\in (0,+\infty)\times S\times \R^\nen:x\in A'_{r,t}\} \et \{(r,t,x)\in (0,+\infty)\times S\times \R^\nen:x\in B_t'\}.$$ Clearly, the cells of $\C_{r,t}$  (see Remark \ref{rem_fibres_cell_dec} for $\C_{r,t}$) that are included in $A'_{r,t}$ or $B_t'$ are graphs of functions induced by the $\xi_{i,r,t}$'s.
   Note  that for each $t\in S$ and $r>0$ there must be $C\in \C$ (depending on $r,t$)  such that $C_{r,t}\subset B_t'$  and $\hn^{\nen-1} (C_{r,t}) \ge \frac{\hn^{\nen-1} (B_t')}{N}$,  where $N$ is the number of cells of $\C$. By
 (\ref{eq_hyp_S2}), this means that for each $t\in S$ satisfying (\ref{eq_second}) and $r<\delta(t)$ we have $\hn^{\nen-1} (C_{r,t}) \ge\frac{\alpha}{nN\hn^k(\bou(0,1))} $, which, since $C_{r,t}$ is included in the union of the graphs of the $\xi_{i,r,t}$'s, implies that  $\hn^{\nen-1} (\pi_\nen(C_{r,t}))\ge \frac{\alpha}{\kappa nN\hn^k(\bou(0,1))} $ (here $\pi_\nen:\R^\nen\to \R^{\nen-1}$ is the canonical projection) for some constant $\kappa$ that just depends on the Lipschitz constants of the $\xi_{i,r,t}$'s.

 By Lemma \ref{lem_boules}, we deduce that for each  $t$ satisfying (\ref{eq_second}) and $r<\delta(t)$, $\pi_\nen(C_{r,t})$ contains a ball $\bou_{\R^{\nen-1}}(a(r,t),\sigma)$ for some $\sigma>0$ (independent of $t$ and $r$)  and some $a(r,t)\in \R^{\nen-1}$.  The argument we used in (\ref{ineq_A_t_B_t})  applies word for word (just replacing  $\R^n\setminus \omd$ with $\delta \omd$) to show that $d_\hn (A_{r,t}', B_t')<\sigma$ for all $r< \delta(t)$, for some $\delta\in \D^+(S)$, which entails that $\pi_\nen(A'_{r,t})$ must meet $ \pi_\nen(C_{r,t})$ for such $r,t$. Since $\C_{r,t}$ is for each $r,t$ a cell decomposition  compatible with $A_{r,t}'$,  this implies that $\pi_{\nen}(C_{r,t})$ entirely fits  in $\pi_\nen (A_{r,t}')$.
 Hence,  $A_{r,t}'$ is above $\bou_{\R^{\nen-1}}(a(r,t),\sigma)$ the graph of a  Lipschitz function (induced by the  $\xi_{i,r,t}$'s) whose Lipschitz constant can be bounded independently of $r,t$.
 We thus can apply classical arguments for proving Poincar\'e inequality on Lipschitz domains. As well-known, the constant of this inequality can be bounded in terms of the Lipschitz constant of the domain and the radius of the ball on which the function is defined. This yields (\ref{eq_sphere_loc}).

 Now, thanks to Lemma \ref{lem_boules}, (\ref{eq_sphere_loc}) implies (\ref{ineq_short_sph}) for  $v\in  W^{1,p}(\sph^\nen, A_{r,t,reg}')$, and we may write (\ref{ineq_vtr}) to complete the proof of  (\ref{ineq_vtr0}) in the second  case.
 
%  from step \ref{krok_vtr_h2} (taking $\alpha_2:=\alpha$).
%\end{proof}
 % for all $t$ such that 
% \begin{equation}\label{ineq_hnbtr_pf_short}
% 	\hn^{n-k-2}(\cbt_t (\R^n\setminus \omd))\ge \alpha,	
% \end{equation}
% where $\alpha$ is provided by assumption (\ref{ass_cone_tangent}).

%  is such that (\ref{ineq_hnbtr_pf_short}) holds, because $\bou(x_t,\alpha')\subset\cbt_t (\R^n\setminus \omd) $
To prove (\ref{ineq_vshort}), it then suffices to integrate (\ref{ineq_vtr0}). Namely, since $\pi_S$ is the identity on $S$ and $T_x V^\delta _r= (x-\pi_S(x))^\perp\supset T_{\pi_S(x)} S$ for all $x\in V^\delta_r$, there is $\delta\in \D^+(S)$ such that $\jac \pi_{S|V_r^\delta} (x)$ is bounded below away from zero,  which entails
 \begin{equation}\label{ineq_vtr2}
 	||u||_{L^p(V^\delta_r)}^p\overset{(\ref{eq_coarea})}{\lesssim} \int_{S}  ||u||^p_{L^p(V_{r,t})}   dt\overset{(\ref{ineq_vtr0})}{\lesssim} r^p\int_{S}  ||\nabla u||^p_{L^p(\frac{1}{r}\cdot V_{r,t})}  dt\overset{(\ref{eq_coarea})}{\lesssim} r^p||\nabla u||_{L^p(V^\delta_r)}^p,
 \end{equation}
 where the constant is independent of $r$.
\end{proof}

\begin{proof}[Proof of (\ref{ineq_vshort_reg})] 
%	Let $\pi_S:U_S\to S$ be the retraction given by the closest point, and 
	 Set $g_\nu(x)=\nu x+(1-\nu)\pi_S(x)$, $\nu\in [0,1]$, $x\in U_S$, and notice that, by Minskowski's integral inequality, we have for $\delta \in \D^+(S)$, $u\in W^{1,p}(\omd,\pa \omd)$, and $\eta>0$:
	 % then we have for $\delta \in \D^+(S)$ and $\eta>0$: $|\frac{\pa g_t(x)}{\pa t}|\le \eta$ as soon as .
\begin{equation*}
		||u||_{L^p(V_\eta^\delta)}=\left(\int_{V_\eta^\delta}\left| \int_0^1\frac{d (u\circ g_\nu)}{d\nu}d\nu\right|^p\right)^{1/p}\\
		\le \int_0^1 \left(\int_{V_\eta^\delta}|\nabla u \circ g_\nu|^p\cdot \left|\frac{\pa g_\nu}{\pa \nu}\right|^p\right)^{1/p}d\nu.	\end{equation*}
 Since $ |\frac{\pa g_\nu}{\pa \nu}(x)|\le \eta$ when $d(x,S)=\eta $, this implies
		\begin{equation*}	||u||_{L^p(V_\eta^\delta)}	\le \eta\int_0^1 \left(\int_{V_\eta^\delta}|\nabla u \circ g_\nu|^p\right)^{1/p}d\nu
			\le\eta \left(\int_0^1\int_{V_\eta^\delta}|\nabla u \circ g_\nu|^p\,d\nu\right)^{1/p}\quad \mbox{(by H\"older)}.	\end{equation*}
		As $S$ is an $(n-1)$-dimensional smooth submanifold of $\R^n$, for each $\delta \in \D^+(S)$ small enough, $D_x\pi_S:T_x V^\delta_\eta \to T_{\pi_S(x)}S$ is close  to the identity for all $x\in V_\eta^\delta$, for all $\eta>0$. Hence, for $\delta$ small enough, $g_\nu$ induces an embedding of $V^\delta _\eta$ into $V^\delta _{\nu\eta}$ such that $\jac g_{\nu|V^\delta_\eta}(x) \ge \frac{1}{2}$. Applying the coarea formula (\ref{eq_coarea}) in the right-hand-side of the above inequality, we get
				\begin{equation*}	 ||u||_{L^p(V_\eta^\delta)}	 \le2\eta  \left(\int_0^1\int_{V_{\nu\eta}^\delta}|\nabla u(x) |^pdx\,d\nu\right)^{1/p}\quad  \\
		\le 2\,\eta^{1-\frac{1}{p}} \, \left(\int_0^\eta\int_{V_{\zeta}^\delta}|\nabla u(x) |^pdx\,d\zeta\right)^{1/p},
\end{equation*}
where we have set $\zeta(\nu):=\nu\eta$. Applying again  the coarea formula  (with now the function $d(\cdot,S)$) gives the desired inequality.
\end{proof}

\end{document}